\documentclass[letterpaper]{article}
\usepackage{aaai2026}
\usepackage{times}
\usepackage{helvet}
\usepackage{courier}
\usepackage{kotex}
\usepackage[hyphens]{url}
\usepackage{graphicx}
\usepackage{natbib}
\usepackage{caption}
\ifx\pdfoutput\undefined
\else
  \ifnum\pdfoutput>0
  \fi
\fi

\usepackage{amsmath,amssymb,amsfonts,amsthm,mathtools}
\usepackage{booktabs}
\usepackage{array}
\usepackage[colorlinks=true,linkcolor=blue,citecolor=blue,urlcolor=blue]{hyperref}
\usepackage[capitalize,noabbrev]{cleveref}

\numberwithin{equation}{section}
\newtheorem{definition}{Definition}[section]
\newtheorem{assumption}{Assumption}[section]
\newtheorem{proposition}{Proposition}[section]
\newtheorem{lemma}{Lemma}[section]
\newtheorem{theorem}{Theorem}[section]

\newcommand{\R}{\mathbb{R}}
\newcommand{\jsr}{\rho}
\newcommand{\co}{\operatorname{co}}
\newcommand{\diag}{\operatorname{diag}}

\title{Spectral Analysis of Heavy-Ball Q-Value Iteration}
\author{Donghwan Lee}
\affiliations{
Department of Electrical Engineering, Korea Advanced Institute of Science and Technology (KAIST)\\
Daejeon 34141, South Korea\\
donghwan@kaist.ac.kr
}

\begin{document}
\maketitle

\begin{abstract}
We study the convergence and acceleration of Q-value iteration (QVI) with
heavy-ball momentum. Although acceleration of value iteration has been
studied extensively, there has been less work on the acceleration of heavy-ball
QVI for control tasks. We analyze heavy-ball QVI from the viewpoint of
switched linear system (SLS) theory and the joint spectral radius (JSR). First, we
convert heavy-ball QVI into an exact SLS and interpret its
convergence through the JSR. We also give JSR-based conditions under
which heavy-ball QVI can be faster than standard QVI.
\end{abstract}

\begin{center}
\end{center}

\section{Introduction}
\label{sec:introduction}

Q-value iteration (QVI) is a basic fixed-point method for discounted Markov
decision process (MDP) control~\citep{puterman1994markov,bertsekas1996neuro}. Its
standard convergence analysis is based on the contraction property of the Bellman
optimality operator~\citep{puterman1994markov,bertsekas1996neuro}. Many
acceleration methods have been developed for value iteration and related Bellman
iterations~\citep{bertsekas1995rankone,farahmand2021pid,goyal2022firstorder},
but the acceleration of QVI with a heavy-ball momentum term is less understood
for control tasks~\citep{farahmand2021pid,goyal2022firstorder,weng2019momentum}.

The main difficulty is that Bellman optimality QVI is not a single fixed linear
system. Because the greedy action can change with the current Q-function, the
linear part of the error recursion switches across policy-dependent modes. After
heavy-ball momentum is added, the method becomes an augmented switched linear
system (SLS). Consequently, an eigenvalue analysis of one fixed policy does not
by itself certify global worst-case acceleration.

We study the heavy-ball QVI recursion~\citep{farahmand2021pid,goyal2022firstorder,weng2019momentum} given by the following update:
\begin{align}
  Q_{k+1}=Q_k+\alpha(F(Q_k)-Q_k)+\eta(Q_k-Q_{k-1}),
  \label{eq:intro-hb-update}
\end{align}
where $F$ is the Bellman optimality operator, $\alpha\in (0,1]$ is a
relaxation parameter (or a step-size), and $\eta>0$ is the heavy-ball momentum parameter. We
rewrite this recursion exactly as an SLS~\citep{liberzon2003switching,lin2009stability,shorten2007stability}
around the optimal Q-function. This representation lets us interpret convergence
and acceleration through the joint spectral radius (JSR)~\citep{rota1960note,jungers2009joint},
which measures the worst-case product growth rate of the SLS
modes~\citep{jungers2009joint,liberzon2003switching}.
In an SLS, the active system matrix changes with the switching signal. Hence,
the eigenvalues or spectral radius of a single fixed matrix do not generally
characterize worst-case product growth. We therefore use the JSR as the main
rate notion. Our acceleration analysis focuses on the common eigenvector
direction shared by the switching matrix family; it does not by itself
characterize the remaining transverse directions. Nevertheless, this analysis is
a first step toward understanding heavy-ball QVI for control tasks and provides
intuition for sharper future analyses.

\section{Related Work}
\label{sec:related_works}

Momentum-based modifications of value iteration (VI) have been studied from
several perspectives. \citet{farahmand2021pid} interpret VI as a feedback-control system in their proportional-integral-derivative (PID) accelerated VI framework and introduce proportional-derivative (PD), proportional-integral (PI), and PID variants. In fixed-policy settings, the PD/heavy-ball term changes the eigenvalues of the error dynamics of VI through a second-order companion polynomial associated with each eigenvalue of the transition matrix. \citet{goyal2022firstorder} also study momentum value computation and accelerated value iteration, connecting VI with first-order optimization methods and analyzing momentum-type VI updates for policy evaluation. Momentum-based accelerated Q-learning has also been studied in a stochastic approximation setting~\citep{weng2019momentum,bowen2021finitetime}. We analyze heavy-ball QVI as an SLS and compare its JSR against the JSR of standard QVI. Unlike much of the existing momentum-VI literature, which focuses mainly on policy evaluation or on algorithmic variants different from the update in~\Cref{eq:intro-hb-update}, this paper treats the control setting for the simple constant-stepsize heavy-ball QVI recursion and studies it through an exact SLS representation and a JSR-based rate comparison.

Other acceleration mechanisms for dynamic programming iterations are related but technically different. Adaptive relaxation and lookahead methods for VI were studied by~\citet{herzberg1994accelerating,shlakhter2010acceleration}; \citet{geist2018anderson} apply Anderson acceleration to dynamic programming iterations; \citet{vieillard2020momentum} propose Momentum VI, which averages successive $Q$-functions rather than using the heavy-ball QVI update analyzed here; and \citet{lee2023anchoring} give another provably accelerated Bellman-error mechanism for both Bellman consistency and optimality operators. Speedy Q-learning is also relevant because \citet{azar2011speedy} use two successive Q-estimates to obtain faster finite-sample bounds in a model-free setting. These works support the broader idea that past iterates or modified Bellman updates can accelerate Bellman-type recursions, but their proofs and algorithms differ from the SLS JSR analysis of heavy-ball QVI given here.

The slow constant direction of discounted dynamic programming iterations has a
long history. Rank-one correction and extrapolation methods exploit the dominant
stochastic matrix direction to improve VI behavior~\citep{bertsekas1995rankone}. Recent rank-one or deflation-based methods also use dominant-direction information, including rank-one modified VI and deflated dynamics VI~\citep{kolarijani2025rankone,lee2025deflateddynamics}. More recent deflation-based approaches remove or quotient out the component along the common eigenvector in order to expose faster
transverse dynamics; see, for example, the switching-geometry analysis of
deflated QVI by~\citet{lee2026deflated}. The present paper is
related in that every standard QVI switching matrix treats the constant-vector
direction in the same way. The mechanism studied here is different: heavy-ball momentum does not deflate or
remove this direction, but turns it into a two-dimensional companion block. The
main issue is therefore to combine the exact common-eigenvector improvement condition
with a projected JSR condition that controls the remaining policy-dependent
SLS dynamics.

The JSR is a standard measure of worst-case exponential growth for products of
matrices under arbitrary switching; see, for example,~\citep{jungers2009joint}. VI updates naturally lead to SLS representations because the greedy action
can change with the iterate. The SLS analysis used here is close to the SLS representation of Q-learning by~\citet{lee2026lyapunovcertified}. The projected error viewpoint is also related to recent
JSR analyses of QVI geometry, including~\citet{lee2026geometry,lee2026deflated}. Compared with these works, the present paper
focuses on the augmented heavy-ball SLS dynamics.

\section{Preliminaries}
\label{sec:preliminaries}
\subsection{Notation}
\label{sec:notation}

The set of real numbers is denoted by $\R$; $\R^m$ is the $m$-dimensional
Euclidean space; and $\R^{m\times n}$ is the set of all $m\times n$ real
matrices. For a matrix $A$, $A^\top$ denotes its transpose. The identity
matrix is denoted by $I$. For vectors, $e_i$ is the $i$th standard basis vector,
with dimension clear from context, and $\otimes$ denotes the Kronecker product.
For a finite set $\mathcal S$, $|\mathcal S|$ denotes its cardinality. For finite tabular state and action sets, set
\begin{align*}
  n:=|\mathcal S||\mathcal A|.
\end{align*}
We write $\Delta_m:=\left\{q\in\R^m:q_i\geq0,\ \sum_{i=1}^m q_i=1\right\}$ for the probability simplex in $\R^m$. For a finite matrix family
$\mathcal H=\{A_1,\ldots,A_N\}$, $\co(\mathcal H)
:=\left\{\sum_{i=1}^N\lambda_i A_i:
\lambda_i\geq0,\ \sum_{i=1}^N\lambda_i=1\right\}$ denotes its convex hull.
We also use standard matrix notation that appears repeatedly below. For a
vector $x$, $\|x\|_2$ is the Euclidean norm. For a square matrix $A$,
$\rho(A)$ denotes its ordinary spectral radius, while $\rho(\mathcal H)$ denotes
the JSR of a switching family once the JSR is defined below. For a symmetric
matrix $B$, $B\succ0$ means that $B$ is positive definite.

\subsection{Switching Linear System}
\label{sec:switching_systems}

The stability certificates used later are stated in the language of switched linear systems (SLSs), so we first recall the basic model before specializing it to the
Bellman-induced switching families. Let us consider the discrete-time switching
affine system~\citep{liberzon2003switching,lin2009stability,shorten2007stability}
\begin{align*}
  x_{k+1}=A_{\sigma_k}x_k+b_{\sigma_k},
\end{align*}
where each index $i\in\{1,2,\ldots,M\}$, equivalently each affine pair
$(A_i,b_i)$, is called a \emph{mode}, and $\sigma_k$ is the switching signal that
selects the active mode at time $k$. The matrix $A_{\sigma_k}$ is selected from
the prescribed family $\mathcal H:=\{A_1,A_2,\ldots,A_M\}$, which is called a
\emph{switching family}; $b_{\sigma_k}$ is a mode-dependent affine term. When
$b_{\sigma_k}=0$, the deterministic part reduces to an SLS, $x_{k+1}=A_{\sigma_k}x_k$. The worst-case exponential rate of the SLS
family is characterized by the JSR, defined as follows.
\begin{definition}
\label{def:jsr}
For a bounded set of matrices $\mathcal H\subset\R^{m\times m}$, its joint
spectral radius is
\begin{align*}
\jsr(\mathcal H)
:=
\lim_{k\to\infty}
\sup_{A_1,\ldots,A_k\in\mathcal H}
\|A_k\cdots A_1\|^{1/k}.
\end{align*}
\end{definition}
We note that the JSR is independent of the chosen submultiplicative
norm~\citep{rota1960note,jungers2009joint}. When $\mathcal H$ is finite, the
supremum for each fixed product length is a maximum over products generated by
matrices in $\mathcal H$. For a finite family $\mathcal H$, the notation
$\jsr(\co(\mathcal H))$ means the JSR computed when each factor in a product is
allowed to be any convex combination of matrices in $\mathcal H$.
Throughout the later JSR certificates, $\rho(\mathcal H)$ denotes this same JSR
value when the argument is a switching family.
We use two auxiliary JSR facts throughout. Convex-hull invariance and the block
triangular JSR decomposition are stated in Appendix~\ref{app:auxiliary-facts} as
\Cref{lem:convex-hull-jsr,lem:block-triangular-jsr}.

\subsection{Joint Spectral Radius and Lyapunov Certificates}
\label{sec:joint_spectral_radius}
Before proceeding, we note that most proofs are collected in Appendix~\ref{app:proofs-main-results}; the main text focuses on the constructions, statements, and their implications.
The JSR in~\Cref{def:jsr} turns arbitrary SLS products into a
single worst-case exponential rate. An SLS is uniformly exponentially stable under arbitrary switching if there
exist constants $C\geq1$
and $\eta\in(0,1)$ such that $\|A_{\sigma_{k-1}}\cdots A_{\sigma_0}x\|_2\leq C\eta^k\|x\|_2$ for every horizon $k\geq0$, every initial state $x\in\R^m$, and every switching
sequence. A \emph{common Lyapunov function} for $\mathcal H$ is a positive definite
function that decreases along every mode. In the analysis below, the Bellman
maximum in linear Q-learning induces stochastic-policy switching, and the
Lyapunov functions are built from products of the corresponding mode matrices.
The following finite-family piecewise-quadratic construction~\citep{lee2026lyapunovcertified,hushenzhang2010generating} is the Lyapunov certificate used in the error-recursion arguments below.
It turns the abstract JSR condition into a concrete decrease estimate that can be
reused for each algorithmic family.
\begin{lemma}
\label{lem:common_lyapunov_construction}
Let $\mathcal H=\{A_1,A_2,\ldots,A_M\}\subset\R^{m\times m}$ and fix
$\epsilon>0$ such that $\beta_\epsilon:=\rho(\mathcal H)+\epsilon\in(0,1)$.
For a word $\sigma=(\sigma_1,\ldots,\sigma_k)\in\{1,\ldots,M\}^k$, write
\begin{align*}
A_\sigma:=A_{\sigma_k}\cdots A_{\sigma_1},
\end{align*}
with the convention that the empty word gives $A_\sigma=I$. Define
\begin{align*}
V_\epsilon^\infty(x)
:=\sum_{k=0}^\infty \beta_\epsilon^{-2k}
\max_{\sigma\in\{1,\ldots,M\}^k}\|A_\sigma x\|_2^2,
\qquad x\in\R^m.
\end{align*}
Then $V_\epsilon^\infty$ is finite for every $x$, and there exists
$C_\epsilon>0$ such that
\begin{align*}
\|x\|_2^2\leq V_\epsilon^\infty(x)\leq C_\epsilon\|x\|_2^2,
\qquad \forall x\in\R^m.
\end{align*}
The function $p_\epsilon(x):=\sqrt{V_\epsilon^\infty(x)}$ is a norm on
$\R^m$, and every mode satisfies
\begin{align*}
p_\epsilon(A_i x)\leq\beta_\epsilon p_\epsilon(x),
\qquad \forall x\in\R^m,
\qquad i=1,\ldots,M.
\end{align*}
\end{lemma}

Throughout the sequel, whenever this construction is applied to a switching
family with JSR less than one, we call the resulting $V_\epsilon^\infty$ a \emph{JSR
Lyapunov function} for that family, and we call the associated norm $p_\epsilon$
a \emph{JSR Lyapunov norm}.
The same Lyapunov argument will be used for the error recursions considered
below. The following lemma is the common bridge from a JSR bound to convergence of the corresponding error recursion.
\begin{lemma}
\label{lem:standard_jsr_convergence}
\label{lem:basic_jsr_convergence}
Let $\mathcal H=\{A_1,\ldots,A_M\}\subset\R^{n\times n}$ be finite and suppose
$\rho(\mathcal H)<1$. Let us consider any error recursion
\begin{align*}
x_{k+1}=A_kx_k,
\qquad A_k\in\co(\mathcal H),
\qquad k\in\{0,1,\ldots\}.
\end{align*}
Then, for every $\epsilon>0$ such that
$\beta_\epsilon:=\rho(\mathcal H)+\epsilon<1$, the Lyapunov function
$V_\epsilon^\infty$ and norm $p_\epsilon$ from~\Cref{lem:common_lyapunov_construction}, applied to
$\mathcal H$, satisfy
\begin{align*}
V_\epsilon^\infty(x_{k+1})
\leq
\beta_\epsilon^2 V_\epsilon^\infty(x_k),
\qquad
p_\epsilon(x_{k+1})\leq \beta_\epsilon p_\epsilon(x_k).
\end{align*}
Consequently, if $C_\epsilon$ is the constant from~\Cref{lem:common_lyapunov_construction}, then
\begin{align*}
p_\epsilon(x_k)&\leq \beta_\epsilon^k p_\epsilon(x_0),\\
\|x_k\|_2
&\leq \beta_\epsilon^k p_\epsilon(x_0)\\
&\leq \sqrt{C_\epsilon}\,\beta_\epsilon^k\|x_0\|_2,
\end{align*}
and hence $x_k\to0$.
\end{lemma}

\subsection{Discounted Markov Decision Processes}
\label{sec:discounted_mdps}

We consider a finite discounted Markov decision process (MDP)~\citep{puterman1994markov,bertsekas1996neuro} with state-space
$\mathcal S=\{1,\ldots,|\mathcal S|\}$, action-space
$\mathcal A=\{1,\ldots,|\mathcal A|\}$, transition probability
$P(s'\mid s,a)$, real-valued one-step reward $r(s,a,s')$, expected reward $R(s,a):=\sum_{s'\in\mathcal S}P(s'\mid s,a)r(s,a,s')$, and discount factor $\gamma\in(0,1)$. State-action functions are viewed as
vectors in $\R^{n}$ using the action-block ordering $(1,1),(2,1),\ldots,(|\mathcal S|,1),
(1,2),(2,2),\ldots,(|\mathcal S|,|\mathcal A|)$. All matrices and vectors indexed by state-action pairs use this ordering. Define
\begin{align*}
P:=
\begin{bmatrix}
P_1\\
\vdots\\
P_{|\mathcal A|}
\end{bmatrix}
\in\R^{n\times |\mathcal S|},
\qquad
R:=
\begin{bmatrix}
R(\cdot,1)\\
\vdots\\
R(\cdot,|\mathcal A|)
\end{bmatrix}
\in\R^{n},
\end{align*}
where $P_a=P(\cdot\mid\cdot,a)\in\R^{|\mathcal S|\times|\mathcal S|}$.
Let $\Theta$ denote the set of deterministic stationary policies
$\pi:\mathcal S\to\mathcal A$. For any stochastic policy
$\mu:\mathcal S\to\Delta_{|\mathcal A|}$, define
\begin{align*}
\Pi^\mu:=
\begin{bmatrix}
\mu(1)^\top\otimes e_1^\top\\
\mu(2)^\top\otimes e_2^\top\\
\vdots\\
\mu(|\mathcal S|)^\top\otimes e_{|\mathcal S|}^\top
\end{bmatrix}
\in\R^{|\mathcal S|\times n}.
\end{align*}
For a deterministic policy $\pi\in\Theta$, the same notation $\Pi^\pi$ is used by
identifying $\pi(s)$ with its one-hot encoding.
For
$Q\in\R^{n}$, define
\begin{align*}
V_Q(s)&:=\max_{a\in\mathcal A}Q(s,a),\\
V_Q&:=(V_Q(1),\ldots,V_Q(|\mathcal S|))^\top.
\end{align*}
The Bellman optimality operator is
\begin{equation*}
  F(Q):=R+\gamma P V_Q.
\end{equation*}
The Bellman optimal Q-function $Q^\star\in\R^n$ is the unique fixed point of the
Bellman optimality operator $Q^\star=F(Q^\star)$.
For a tie-broken greedy policy $\pi_Q$ satisfying
\begin{align*}
  \pi_Q(s)\in\operatorname*{arg\,max}_{a\in\mathcal A}Q(s,a),
  \qquad s\in\mathcal S,
\end{align*}
one has $V_Q=\Pi^{\pi_Q}Q$ and
\begin{equation*}
  F(Q)=R+\gamma P\Pi^{\pi_Q}Q.
\end{equation*}
For every stochastic policy $\mu$, the matrix $P\Pi^\mu$ is row-stochastic and
satisfies
\begin{align}
  P\Pi^\mu\mathbf{1}=\mathbf{1},\label{eq:policy-selector-row-stochastic}
\end{align}
where $\mathbf{1}$ is the all-ones vector with dimension clear from context. The
Bellman-difference selector and the convex-hull relation between stochastic and
deterministic policy selectors are collected in \Cref{lem:bellman_difference_selector} in Appendix.

\section{Standard QVI as the Fixed-Point Benchmark}
\label{sec:standard_q_baseline}

In this paper, we consider the $\alpha$-relaxed QVI update, as a comparison benchmark, given by
\begin{equation}
\begin{aligned}
Q_{k+1}
&=Q_k+\alpha(F(Q_k)-Q_k)\\
&=(1-\alpha)Q_k+\alpha F(Q_k),
\end{aligned}
\label{eq:standard-q-update}
\end{equation}
where $0<\alpha\le1$ can be seen as a relaxation parameter or a step-size.
This is the relaxed version of standard QVI. When $\alpha=1$,~\Cref{eq:standard-q-update} reduces to the usual standard QVI. As in the SLS representation of Q-learning by~\citet{lee2026lyapunovcertified}, we can show that the standard QVI error admits an exact SLS form.
\begin{lemma}
\label{lem:standard-q-error-system}
Using~\Cref{lem:bellman_difference_fixed_point_selector} in Appendix, the error
dynamics around $Q^\star$ are expressed as
\begin{equation*}
  Q_{k+1}-Q^\star
  =A_{\mu_k}^{\rm QVI}(Q_k-Q^\star),
\end{equation*}
where for any stochastic policy $\mu$, $A_\mu^{\rm QVI}$ is defined as
\begin{align*}
  A_\mu^{\rm QVI}:=(1-\alpha)I+\alpha\gamma P\Pi^\mu.
\end{align*}
\end{lemma}

Since the error evolves as an SLS, its worst-case convergence rate is naturally measured by the JSR of its SLS family. The JSR gives the maximal asymptotic exponential growth rate over all switching products, and it is independent of the particular norm used in its definition~\citep{rota1960note,jungers2009joint}.
We introduce the corresponding switching family
\begin{equation*}
  \mathcal A^{\rm QVI}:=\left\{A_\pi^{\rm QVI}:\pi\in\Theta\right\}.
\end{equation*}
Note that this switching family is finite and only collects matrices corresponding to deterministic policies.
Therefore, for every stochastic selector $\mu_k$, one has
$A_{\mu_k}^{\rm QVI}\in\co(\mathcal A^{\rm QVI})$ by the convex-hull identity in \Cref{lem:bellman_difference_selector} in Appendix.
The next proposition shows that the JSR of $\co(\mathcal A^{\rm QVI})$ is identical to the JSR of the finite switching family $\mathcal A^{\rm QVI}$.
\begin{proposition}
\label{prop:standard_q_jsr}
For the standard QVI family, one has
\begin{equation}
  \rho(\mathcal A^{\rm QVI})
  =\rho(\co(\mathcal A^{\rm QVI}))
  =1-\alpha(1-\gamma).
  \label{eq:q-jsr}
\end{equation}
\end{proposition}
We also note that if $0<\alpha\le1$, then $\rho(\mathcal A^{\rm QVI})=1-\alpha(1-\gamma)<1$. We therefore impose this range as an assumption.
\begin{assumption}
\label{ass:qvi-alpha-stability}
Throughout the paper, we assume that the step-size satisfies $0<\alpha\le1$.
\end{assumption}
A key consequence of the row-stochastic structure in~\Cref{eq:policy-selector-row-stochastic} is that the vector $\mathbf{1}$ is a common eigenvector for every QVI mode:
\begin{equation*}
  A_\mu^{\rm QVI}\mathbf{1} =\bigl(1-\alpha(1-\gamma)\bigr)\mathbf{1},
\end{equation*}
for every stochastic policy $\mu$. This common direction is exactly the ambient-space bottleneck as emphasized
by~\citet{lee2026deflated}: convergence of the iterate $Q_k$ can be governed by
the component along this common eigenvector $\mathbf{1}$ even when the transverse projected dynamics are faster.
The next lemma presents the common eigenvector direction of the standard QVI SLS
family and the scalar recursion induced on that direction.
\begin{lemma}
\label{lem:standard-qvi-common-eigenvector}
For every stochastic policy $\mu$, one has
\begin{equation*}
  A_\mu^{\rm QVI}\mathbf{1}
  =\bigl(1-\alpha(1-\gamma)\bigr)\mathbf{1}.
\end{equation*}
Consequently, $\operatorname{span}\{\mathbf{1}\}$ is a common invariant subspace
of the standard QVI modes. If $Q_k-Q^\star=a_k\mathbf{1}$, then
$Q_{k+1}-Q^\star=a_{k+1}\mathbf{1}$ with
\begin{equation*}
  a_{k+1}=\bigl(1-\alpha(1-\gamma)\bigr)a_k.
\end{equation*}
\end{lemma}

We next introduce the projection used to separate the common eigenvector
direction $\mathbf{1}$ from its orthogonal components for the standard QVI modes.
The same projection will be reused for the heavy-ball SLS family in~\Cref{sec:global_jsr_improvement}.
\begin{definition}
Let
\begin{equation}
  \Pi_\perp:=I-\frac{1}{n}\mathbf{1}\mathbf{1}^\top,
  \label{eq:orthogonal-projection}
\end{equation}
be the orthogonal projection onto $\operatorname{span}\left\{\mathbf{1}\right\}^\perp$. Let
$U\in\R^{n\times(n-1)}$ have orthonormal columns spanning
$\operatorname{span}\left\{\mathbf{1}\right\}^\perp$, so that $U^\top U=I_{n-1}$,
$U^\top\mathbf{1}=0$, and $UU^\top=\Pi_\perp$.
The projected standard QVI matrices and family are
\begin{equation}
  \bar A_\pi^{\rm QVI}:=U^\top A_\pi^{\rm QVI}U,
  \qquad
  \bar{\mathcal A}^{\rm QVI}:=\left\{\bar A_\pi^{\rm QVI}:\pi\in\Theta\right\}.
  \label{eq:projected-ql-family}
\end{equation}
For a stochastic policy $\mu$, the same formula defines $\bar A_\mu^{\rm QVI}$.
\end{definition}

Next, we show that the projected error $z_k:=U^\top(Q_k-Q^\star)$ is itself an SLS for standard QVI, and its active matrix is exactly the projected matrix defined in~\Cref{eq:projected-ql-family}.
\begin{lemma}
\label{lem:projected-error-systems}
Let $z_k:=U^\top(Q_k-Q^\star)$. For any stochastic-policy selector $\mu_k$, the
standard projected error satisfies
\begin{align*}
  z_{k+1}=\bar A_{\mu_k}^{\rm QVI}z_k.
\end{align*}
\end{lemma}

With this projected SLS representation, the JSR of the standard QVI SLS in~\Cref{lem:standard-q-error-system} can be analyzed by using a triangular decomposition that separates the common eigenvector direction from the transverse projected dynamics in~\Cref{lem:projected-error-systems,eq:projected-ql-family}.
\begin{lemma}
\label{lem:qvi-common-eigenvector-projected-jsr-decomposition}
For the switching family $\mathcal A^{\rm QVI}$ of the standard QVI and the
corresponding projected switching family $\bar{\mathcal A}^{\rm QVI}$ defined in
\Cref{eq:projected-ql-family}, we have
\begin{equation*}
  \rho(\mathcal A^{\rm QVI})
  =
  \max\left\{1-\alpha(1-\gamma),\rho(\bar{\mathcal A}^{\rm QVI})\right\}.
\end{equation*}
\end{lemma}

We next compare the JSR of the projected SLS family,
$\rho(\bar{\mathcal A}^{\rm QVI})$, with the JSR of the original SLS family
$\rho(\mathcal A^{\rm QVI})$.
\begin{lemma}
\label{lem:projected-ql-weak-bound}
Under the standing discounted and relaxation assumptions, the projected standard
QVI family satisfies
\begin{align*}
  \rho(\bar{\mathcal A}^{\rm QVI})
  \le \rho(\mathcal A^{\rm QVI})
  =\bigl(1-\alpha(1-\gamma)\bigr).
\end{align*}
\end{lemma}

Therefore the weak inequality
$\rho(\bar{\mathcal A}^{\rm QVI})\le\bigl(1-\alpha(1-\gamma)\bigr)$ always holds.
The next section presents the heavy-ball QVI with a momentum term.

\section{Heavy-Ball QVI}
\label{sec:heavy_ball_bellman_residual}

This section introduces the heavy-ball QVI recursion and derives
its exact SLS representation. The update is given as
\begin{equation}
  Q_{k+1}
  =Q_k+\alpha(F(Q_k)-Q_k)+\eta(Q_k-Q_{k-1}).
  \label{eq:simple-heavy-ball}
\end{equation}
Equivalently, it can be written as
\begin{align*}
  Q_{k+1}=(1-\alpha+\eta)Q_k-\eta Q_{k-1}+\alpha F(Q_k).
\end{align*}
The parameter $\alpha$ is the step-size and $\eta$ is the heavy-ball
momentum parameter. When $\eta=0$, the update reduces to the standard QVI update in~\Cref{eq:standard-q-update}. A natural initialization is $Q_{-1}=Q_0$, so that the first
step coincides with the standard QVI step.
Define the augmented first-order mapping $g:\R^n\times\R^n\to\R^n\times\R^n$ by
\begin{equation*}
  g(x,y):=
  \begin{bmatrix}
    (1-\alpha+\eta)x-\eta y+\alpha F(x)\\
    x
  \end{bmatrix}.
\end{equation*}
Then, the heavy-ball recursion in~\Cref{eq:simple-heavy-ball} is equivalently
\begin{equation*}
\begin{aligned}
\begin{bmatrix}Q_{k+1}\\Q_k\end{bmatrix}
&=g(Q_k,Q_{k-1}).
\end{aligned}
\end{equation*}
Therefore, a fixed point of the heavy-ball algorithm in~\Cref{eq:simple-heavy-ball} is a pair $(x,y)$ satisfying
$g(x,y)=(x,y)$. The following lemma verifies that the fixed point of $g$ is
exactly the fixed point of the Bellman equation copied into both coordinates.
\begin{lemma}
\label{lem:hb-fixed-point-unchanged}
The fixed point of the mapping $g$ is $(Q^\star,Q^\star)$.
\end{lemma}

Therefore, heavy-ball QVI changes the transient SLS dynamics, but not the fixed
point of the Bellman equation. Next, subtracting the augmented fixed point gives the
first-coordinate error recursion
\begin{equation}
\begin{aligned}
Q_{k+1}-Q^\star
  &=(1-\alpha+\eta)(Q_k-Q^\star)\\
  &\quad-\eta(Q_{k-1}-Q^\star)
    +\alpha\{F(Q_k)-F(Q^\star)\}.
\end{aligned}
\label{eq:error-dynamic1}
\end{equation}
In the next lemma, we prove that this error evolution is expressed as an SLS.
\begin{lemma}
\label{lem:hb-exact-sls}
The error of the heavy-ball QVI in~\Cref{eq:error-dynamic1} satisfies the SLS recursion
\begin{equation}
  \begin{bmatrix}Q_{k+1}-Q^\star\\Q_k-Q^\star\end{bmatrix}
  =A_{\mu_k}^{\rm HB}
  \begin{bmatrix}Q_k-Q^\star\\Q_{k-1}-Q^\star\end{bmatrix},
  \label{eq:simple-error}
\end{equation}
where $\mu_k$ is a stochastic policy selector given by~\Cref{lem:bellman_difference_selector} in Appendix and
\begin{align*}
  A_\mu^{\rm HB}:=
  \begin{bmatrix}
    (1-\alpha+\eta)I+\alpha\gamma P\Pi^\mu & -\eta I\\
    I&0
  \end{bmatrix}
\end{align*}
for any stochastic policy $\mu$. The corresponding heavy-ball SLS family is
\begin{equation}
  \mathcal A^{\rm HB}:=\left\{A_\pi^{\rm HB}:\pi\in\Theta\right\}.
  \label{eq:hb-mode-set}
\end{equation}
For every stochastic selector $\mu_k$, one has
$A_{\mu_k}^{\rm HB}\in\co(\mathcal A^{\rm HB})$.
\end{lemma}

As in the standard QVI SLS family, every heavy-ball mode leaves an augmented
common-eigenvector subspace invariant. Define
\begin{equation}
  C_\eta^-:=
  \begin{bmatrix}
    \bigl(1-\alpha(1-\gamma)\bigr)+\eta&-\eta\\
    1&0
  \end{bmatrix}.
  \label{eq:common-eigenvector-companion}
\end{equation}
Then, for every stochastic policy $\mu$ and all $x,y\in\R$,
\begin{align*}
  A_\mu^{\rm HB}
  \begin{bmatrix}x\mathbf{1}\\y\mathbf{1}\end{bmatrix}
  &=
  \begin{bmatrix}
    \{\bigl(1-\alpha(1-\gamma)\bigr)+\eta\}x\mathbf{1}-\eta y\mathbf{1}\\
    x\mathbf{1}
  \end{bmatrix}\\
  &=\left(C_\eta^-\otimes I_n\right)
  \begin{bmatrix}x\mathbf{1}\\y\mathbf{1}\end{bmatrix}.
\end{align*}
Therefore, the standard common invariant eigenvector direction is replaced by the
augmented invariant subspace
\begin{equation}
  \mathcal S_{\mathbf{1}}^{\rm HB}
  :=\left\{\begin{bmatrix}x\mathbf{1}\\y\mathbf{1}\end{bmatrix}:x,y\in\R\right\},
  \label{eq:hb-common-eigenvector-subspace}
\end{equation}
which satisfies
\begin{equation}
  A_\mu^{\rm HB}\mathcal S_{\mathbf{1}}^{\rm HB}
  \subseteq \mathcal S_{\mathbf{1}}^{\rm HB}  \label{eq:hb-common-eigenvector-invariance}
\end{equation}
for every stochastic policy $\mu$. These directions can still be a bottleneck for the heavy-ball JSR. The
next lemma provides the common eigenvector direction and gives the scalar recursion induced on that direction.
\begin{lemma}
\label{lem:hb-common-eigenvector-recursion}
If $Q_k-Q^\star=a_k\mathbf{1}$ and
$Q_{k-1}-Q^\star=a_{k-1}\mathbf{1}$, then
$Q_{k+1}-Q^\star=a_{k+1}\mathbf{1}$, where
\begin{equation}
  a_{k+1}=
  \bigl(\bigl(1-\alpha(1-\gamma)\bigr)+\eta\bigr)a_k
  -\eta a_{k-1}.
  \label{eq:common-eigenvector-scalar-recursion}
\end{equation}
Equivalently, with $C_\eta^-$ defined in~\Cref{eq:common-eigenvector-companion},
\begin{equation*}
  \begin{bmatrix}a_{k+1}\\a_k\end{bmatrix}
  =C_\eta^-
  \begin{bmatrix}a_k\\a_{k-1}\end{bmatrix}.
\end{equation*}
\end{lemma}

The above lemma shows that the standard scalar recursion from~\Cref{lem:standard-qvi-common-eigenvector} is replaced, on the invariant subspace associated with this common eigenvector, by the companion recursion with $C_\eta^-$ from~\Cref{eq:common-eigenvector-companion}. Consequently, the convergence factor $1-\alpha(1-\gamma)$ of the standard QVI SLS on this direction is replaced by the spectral radius $\rho(C_\eta^-)$ of the matrix $C_\eta^-$.
\begin{lemma}
\label{lem:hb-common-eigenvector-spectral-replacement}
On the invariant subspace defined in~\Cref{eq:hb-common-eigenvector-subspace}, the standard scalar convergence factor $1-\alpha(1-\gamma)$ is replaced by the spectral radius $\rho(C_\eta^-)$ of the heavy-ball block. Equivalently, the scalar factor $1-\alpha(1-\gamma)$ is replaced by the roots of the polynomial
\begin{equation*}
\begin{aligned}
\bigl\{1-\alpha(1-\gamma)\bigr\}
&\longmapsto
\{\lambda\in\mathbb C:\\
&\quad
\lambda^2-\bigl(1-\alpha(1-\gamma)+\eta\bigr)\lambda+
\eta=0\}.
\end{aligned}
\end{equation*}
Moreover,
\begin{align*}
  \rho(\mathcal A^{\rm HB})\ge \rho(C_\eta^-).
\end{align*}
\end{lemma}

Consequently, the convergence rate of the common eigenvector direction $\mathbf{1}$ can change under
heavy-ball momentum. It may become faster or slower depending on the parameters
$\eta$, $\alpha$, and $\gamma$, through the roots of the polynomial
in~\Cref{lem:hb-common-eigenvector-spectral-replacement}. The characteristic polynomial of $C_\eta^-$ is
\begin{equation}
  \lambda^2-(\bigl(1-\alpha(1-\gamma)\bigr)+\eta)\lambda+\eta=0.
  \label{eq:hb-common-eigenvector-poly}
\end{equation}
Thus, acceleration along the directions in the invariant subspace defined in~\Cref{eq:hb-common-eigenvector-subspace} is necessary for global JSR improvement.
We now present the exact condition on the momentum parameter $\eta>0$ for the acceleration.
\begin{proposition}
\label{prop:common_eigenvector_acceleration}
For the matrix $C_\eta^-$, we have
\begin{align*}
\rho(C_\eta^-)&<\rho(\mathcal A^{\rm QVI})=1-\alpha(1-\gamma)\\
&\Longleftrightarrow 0<\eta<\rho(\mathcal A^{\rm QVI})^2.
\end{align*}
Consequently, a necessary condition for the heavy-ball JSR to be smaller than
that of standard QVI is
\begin{align*}
  0<\eta<\rho(\mathcal A^{\rm QVI})^2.
\end{align*}
\end{proposition}

The preceding proposition gives a lower-bound obstruction through the matrix $C_\eta^-$: if the recursion in~\Cref{eq:common-eigenvector-scalar-recursion} is not faster than the standard QVI benchmark, then the heavy-ball SLS family cannot have a smaller JSR. A complementary upper bound for the full heavy-ball SLS family is given in Appendix~\ref{app:hb-jsr-upper-bound} as \Cref{prop:hb-jsr-upper-bound}. It gives a simple sufficient condition on $\eta$ under which $\rho(\mathcal A^{\rm HB})<1$ holds automatically.

\section{Global JSR Improvement Certificate}
\label{sec:global_jsr_improvement}

The necessary condition in~\Cref{prop:common_eigenvector_acceleration} derived in the previous section is not sufficient by itself, because the transverse directions outside the invariant subspace in~\Cref{eq:hb-common-eigenvector-subspace} can also determine the JSR. The projection and the projected QVI SLS families were introduced at the end of~\Cref{sec:standard_q_baseline}; here we use them to impose the comparison condition directly on the projected heavy-ball QVI error system. The projection from~\Cref{eq:orthogonal-projection,eq:projected-ql-family} is now applied to the
heavy-ball modes from~\Cref{eq:hb-mode-set}.
\begin{lemma}
\label{lem:projected-hb-error-system}
For each deterministic policy $\pi$, let us define
\begin{equation}
\begin{aligned}
\bar A_\pi^{\rm HB}&:=
  \begin{bmatrix}
    \bar A_\pi^{\rm QVI}+\eta I_{n-1}&-\eta I_{n-1}\\
    I_{n-1}&0
  \end{bmatrix},\\
\bar{\mathcal A}^{\rm HB}&:=
  \left\{\bar A_\pi^{\rm HB}:\pi\in\Theta\right\}.
\end{aligned}
\label{eq:projected-hb-family}
\end{equation}
For a stochastic policy $\mu$, the same formula defines $\bar A_\mu^{\rm HB}$.
Indeed, if $z_k:=U^\top(Q_k-Q^\star)$, then the heavy-ball projected error
satisfies
\begin{equation*}
  \begin{bmatrix}z_{k+1}\\z_k\end{bmatrix}
  =\bar A_{\mu_k}^{\rm HB}
  \begin{bmatrix}z_k\\z_{k-1}\end{bmatrix}.
\end{equation*}
\end{lemma}

The next lemma applies the triangular
separation used in~\Cref{lem:qvi-common-eigenvector-projected-jsr-decomposition} to the augmented heavy-ball SLS family.
\begin{lemma}
\label{lem:hb-common-eigenvector-projected-jsr-decomposition}
For the heavy-ball SLS family $\mathcal A^{\rm HB}$ and the projected
heavy-ball SLS family $\bar{\mathcal A}^{\rm HB}$ defined in~\Cref{eq:hb-mode-set,eq:projected-hb-family}, it holds that
\begin{equation*}
  \rho(\mathcal A^{\rm HB})
  =
  \max\left\{\rho(C_\eta^-),\rho(\bar{\mathcal A}^{\rm HB})\right\}.
\end{equation*}
\end{lemma}

A direct consequence of this decomposition is the following weak bound.
\begin{lemma}
\label{lem:projected-hb-weak-bound}
For the heavy-ball SLS family $\mathcal A^{\rm HB}$ and the projected
heavy-ball SLS family $\bar{\mathcal A}^{\rm HB}$ defined in~\Cref{eq:hb-mode-set,eq:projected-hb-family}, we have
\begin{align*}
  \rho(\bar{\mathcal A}^{\rm HB})\le \rho(\mathcal A^{\rm HB}).
\end{align*}
Thus the projected heavy-ball JSR is bounded by the full heavy-ball JSR.
\end{lemma}

When $\eta=0$, the heavy-ball QVI update in~\Cref{eq:simple-heavy-ball} reduces to
standard QVI in its first block, with the second block only storing the previous
iterate. The same zero-momentum relation holds at the switching-family level:
the lag-augmented matrices have the same asymptotic product growth as the
corresponding standard matrices. We first present this fact for the full
family and then for the projected family.
\begin{lemma}
\label{lem:full-hb-zero-momentum}
At $\eta=0$, the full heavy-ball construction gives
\begin{equation*}
  \rho(\mathcal A^{\rm HB})=\rho(\mathcal A^{\rm QVI}).
\end{equation*}
\end{lemma}

\begin{lemma}
\label{lem:projected-hb-zero-momentum}
At $\eta=0$, the projected heavy-ball construction gives
\begin{equation*}
  \rho(\bar{\mathcal A}^{\rm HB})=\rho(\bar{\mathcal A}^{\rm QVI}).
\end{equation*}
\end{lemma}

We now state a sufficient condition under which heavy-ball QVI has a smaller
JSR than standard QVI.

\begin{theorem}
\label{thm:hb-global-jsr-improvement}
If
\begin{equation}
  \rho(\bar{\mathcal A}^{\rm QVI})<\rho(\mathcal A^{\rm QVI}),
  \label{eq:projected-ql-gap}
\end{equation}
then there exists $\eta_0>0$ such that every
\begin{equation*}
  0<\eta<\min\left\{\rho(\mathcal A^{\rm QVI})^2,\eta_0\right\}
\end{equation*}
satisfies
\begin{equation*}
  \rho(\mathcal A^{\rm HB})
  <\rho(\mathcal A^{\rm QVI}).
\end{equation*}
Consequently, the deterministic heavy-ball QVI error recursion converges
exponentially to zero with a worst-case JSR strictly smaller than that of
standard QVI.
\end{theorem}

Next, we briefly note how the proposed heavy-ball QVI update can be extended to a
stochastic heavy-ball Q-learning setting. A closely related deterministic Q-learning recursion is obtained by weighting the
Bellman residual with the diagonal sampling matrix $D$ and keeping the same
heavy-ball lag term. Under the column-vector convention used throughout this
paper, the update is
\begin{equation*}
  Q_{k+1}=Q_k+\alpha D\{F(Q_k)-Q_k\}+\eta(Q_k-Q_{k-1}).
\end{equation*}
The diagonal matrix $D$ weights the Bellman fixed-point residual, while $\eta(Q_k-Q_{k-1})$ is the
momentum term. This deterministic recursion is a natural starting point for
stochastic reinforcement learning extensions with sampled coordinate updates. However, after
inserting $D$, the common-eigenvector argument used above no longer applies in
general: the corresponding modes need not map the all-ones vector to a scalar
multiple of itself unless the diagonal weights are uniform.
Therefore, the common-eigenvector/quotient-space proof in this paper does not directly
certify such stochastic reinforcement learning variants, and a separate analysis is needed.
These stochastic and coordinate-sampled deflated-momentum variants are therefore left as future topics.

The momentum mechanism above is similar in spirit to the deflated QVI and rank-one modified value-iteration methods studied by~\citet{bertsekas1995rankone,kolarijani2025rankone,lee2025deflateddynamics,lee2026deflated} in that they exploit related low-rank or common-direction structure. The
mechanisms, however, are different. The heavy-ball recursion accelerates the
common eigenvector direction by replacing the scalar recursion with the companion block
$C_\eta^-$, and it also changes the projected lag family
$\bar{\mathcal A}^{\rm HB}$. Thus, depending on the problem and on $\eta$, one can
have
\begin{align*}
  \rho(\bar{\mathcal A}^{\rm HB})<\rho(\bar{\mathcal A}^{\rm QVI}),
\end{align*}
which gives additional acceleration beyond the common-eigenvector improvement. The
projected term is not automatically controlled by the scalar common-eigenvector
argument and must be checked separately. Deflated QVI instead removes the
component along this common eigenvector by a rank-one residual correction and leaves the quotient
trajectory of standard QVI unchanged. In this sense,
heavy-ball momentum can accelerate the common eigenvector direction and may also
accelerate transverse dynamics, whereas deflated QVI removes only
this direction rather than adding a momentum recursion to the projected modes.

An extension to projected Q-value iteration (PQVI) with linear function approximation (LFA) is collected in Appendix~\ref{app:hb-lfa}; the setup and the PQVI construction appear in Appendix~\ref{app:lfa-setup} and Appendix~\ref{app:pqvi-lfa}, respectively. In brief, the same projected-error construction applies when the feature space contains a constant feature: the appendix defines the weighted projection, derives the standard and heavy-ball PQVI SLS families, and gives a small-momentum JSR improvement certificate.

\section{Example}
\label{sec:example}

We now give a small stochastic example that compares standard QVI and heavy-ball
QVI on a simple discounted MDP. The state and action sets are
$\mathcal S=\{1,2\}$ and $\mathcal A=\{1,2\}$, and the discount factor is
$\gamma=0.95$. The transition kernel is stochastic: $P(1\mid 1,1)=0.5027$,
$P(2\mid 1,1)=0.4973$, $P(1\mid 1,2)=0.7461$,
$P(2\mid 1,2)=0.2539$, $P(1\mid 2,1)=0.5617$,
$P(2\mid 2,1)=0.4383$, $P(1\mid 2,2)=0.4418$, and
$P(2\mid 2,2)=0.5582$. The one-step reward is defined by $r(1,1,1)=0.10$, $r(1,1,2)=0.35$,
$r(1,2,1)=0.60$, $r(1,2,2)=0.80$, $r(2,1,1)=0.40$,
$r(2,1,2)=1.00$, $r(2,2,1)=0.95$, and $r(2,2,2)=0.15$.
Starting from $Q_{-1}=Q_0=0$, we compare standard QVI in~\Cref{eq:standard-q-update}
with $\alpha=1$ and heavy-ball QVI in~\Cref{eq:intro-hb-update} with
$\alpha=1$ and $\eta=0.60$. Iterating the Bellman optimality operator to high
accuracy gives $Q^\star(1,1)\approx12.6640$,
$Q^\star(1,2)\approx13.0870$, $Q^\star(2,1)\approx13.1019$, and
$Q^\star(2,2)\approx12.9440$. \Cref{fig:stochastic-two-action-example} plots the error
$\|Q_k-Q^\star\|_\infty$ over 120 iterations. In this example, the error at
iteration 120 is approximately $2.779\times 10^{-2}$ for standard QVI and
$1.211\times 10^{-10}$ for heavy-ball QVI.
\begin{figure}[t]
\centering
\includegraphics[width=\columnwidth]{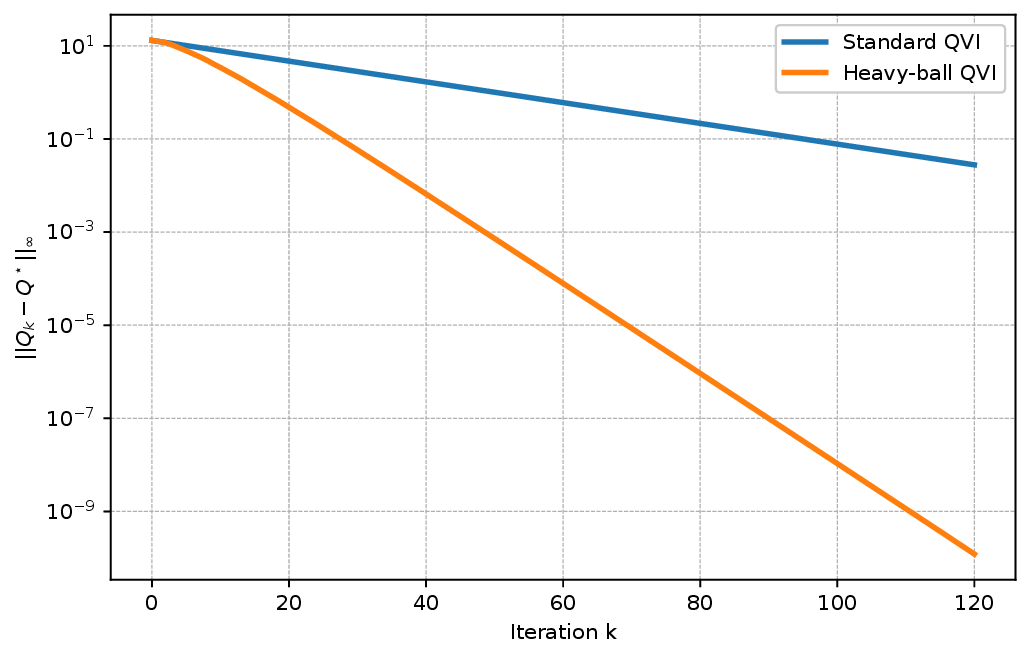}
\caption{Error curves for the MDP instance in~\Cref{sec:example}. The heavy-ball QVI curve uses $\alpha=1$ and $\eta=0.60$.}
\label{fig:stochastic-two-action-example}
\end{figure}

\section{Example}
\label{sec:example2}
We give a tabular control example using the OpenAI Gym/Gymnasium \texttt{Taxi} environment. The
model has 500 states and 6 actions, and we use $\gamma=0.99$. Both methods start
from $Q_{-1}=Q_0=0$ and use the same relaxation parameter $\alpha=0.5$.
Heavy-ball QVI uses $\eta=0.15$, selected from the grid
$\{0,0.05,\ldots,0.35\}$ by the final error after 100 iterations. A reference
$Q^\star$ is computed by exact QVI to high accuracy, and the figure plots
$\|Q_k-Q^\star\|_\infty$.
\begin{figure}[h]
\centering
\includegraphics[width=0.98\columnwidth]{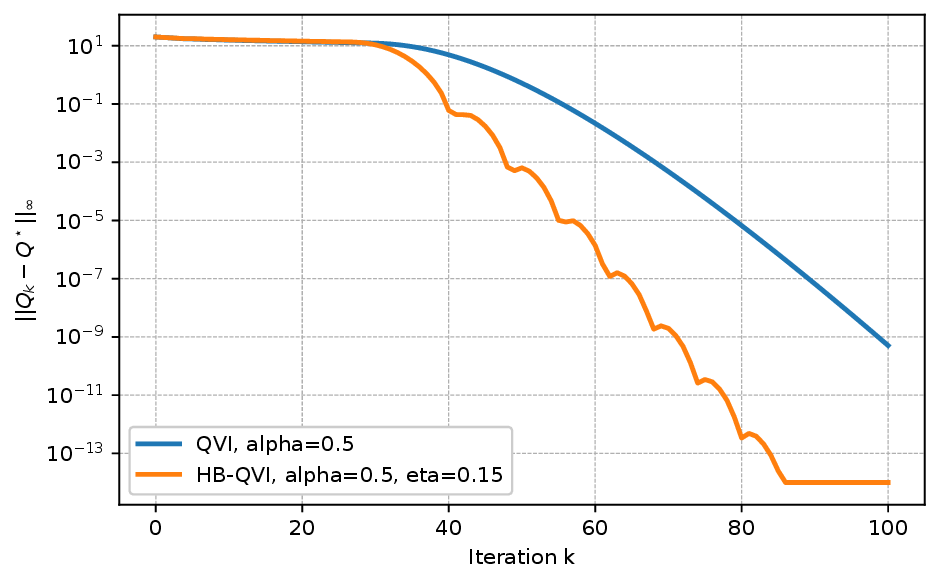}
\caption{Taxi example: QVI and heavy-ball QVI error curves.}
\label{fig:taxi-qvi-hb-example}
\end{figure}
In this example, heavy-ball QVI reduces the error faster after a short transient.
At iteration 100, the infinity-norm error is about $5.06\times10^{-10}$ for QVI
and $7.11\times10^{-15}$ for heavy-ball QVI; both methods recover the same greedy
policy as $Q^\star$.

\section{Conclusion}
\label{sec:conclusion}

This paper studied the convergence and acceleration of heavy-ball QVI through
SLS theory and the JSR. The heavy-ball recursion was rewritten exactly as an SLS around the Bellman
fixed point. Its worst-case asymptotic rate is therefore described by its JSR
rather than by the eigenvalues of a single fixed-policy matrix. The analysis
separates the common eigenvector component from the remaining projected
policy-dependent dynamics: the heavy-ball term can improve the common component,
but global acceleration is certified only when the projected JSR remains below
the standard QVI benchmark.

\clearpage
\onecolumn
\appendix
\section*{\LARGE Appendix}
\section{Auxiliary Facts}
\label{app:auxiliary-facts}

\subsection{JSR Structural Facts}
\label{app:jsr-structural-facts}

\begin{lemma}
\label{lem:convex-hull-jsr}
For a finite matrix family $\mathcal H$, one has
\begin{align*}
  \rho(\co(\mathcal H))=\rho(\mathcal H),
\end{align*}
where the left-hand side is computed over all convex combinations of matrices in
$\mathcal H$.
\end{lemma}

\begin{lemma}
\label{lem:block-triangular-jsr}
Let $\mathcal T$ be a bounded family of block upper triangular matrices of the
form
\begin{align*}
  T_i=\begin{bmatrix}B_i&E_i\\0&D_i\end{bmatrix}.
\end{align*}
Then
\begin{align*}
  \rho(\mathcal T)
  =\max\left\{
    \rho\left(\left\{B_i:T_i\in\mathcal T\right\}\right),
    \rho\left(\left\{D_i:T_i\in\mathcal T\right\}\right)
  \right\}.
\end{align*}
The same statement holds with any finite number of diagonal blocks.
\end{lemma}

\subsection{Bellman Selector Construction}
\label{app:bellman-selector-construction}

\begin{lemma}
\label{lem:bellman_difference_selector}
For any two vectors $Q,\bar Q\in\R^{n}$, there exists a stochastic policy
$\mu_{Q,\bar Q}$ such that
\begin{equation}
  V_Q-V_{\bar Q}=\Pi^{\mu_{Q,\bar Q}}(Q-\bar Q).
  \label{eq:value-difference-selector-general}
\end{equation}
Consequently,
\begin{equation}
  F(Q)-F(\bar Q)=\gamma P\Pi^{\mu_{Q,\bar Q}}(Q-\bar Q).
  \label{eq:bellman-difference-selector-general}
\end{equation}
Moreover, for every stochastic policy $\mu$,
\begin{equation}
  P\Pi^\mu\in\co\left\{P\Pi^\pi:\pi\in\Theta\right\},
  \label{eq:policy-convex-hull}
\end{equation}
and $P\Pi^\mu$ is row-stochastic and satisfies $P\Pi^\mu\mathbf{1}=\mathbf{1}$.
\end{lemma}
\begin{proof}
Fix a state $s$ and write
\begin{align*}
  u_a:=Q(s,a)-\bar Q(s,a),
  \qquad a\in\mathcal A.
\end{align*}
The scalar difference $V_Q(s)-V_{\bar Q}(s)$ lies between
$\min_{a\in\mathcal A}u_a$ and $\max_{a\in\mathcal A}u_a$. Hence it can be
written as a convex combination of the numbers $\{u_a:a\in\mathcal A\}$. Choose
one such convex combination and denote its weights by
$\mu_{Q,\bar Q}(s)\in\Delta_{|\mathcal A|}$. Doing this independently for each
state gives
\begin{align*}
  V_Q(s)-V_{\bar Q}(s)
  =\sum_{a\in\mathcal A}\mu_{Q,\bar Q}(a\mid s)
  \{Q(s,a)-\bar Q(s,a)\},
\end{align*}
for every $s\in\mathcal S$. Stacking the state-wise identities gives
\Cref{eq:value-difference-selector-general}, and multiplying by $\gamma P$ gives
\Cref{eq:bellman-difference-selector-general}.

For the convex-hull statement, write
\begin{align*}
  \lambda_\pi:=\prod_{s\in\mathcal S}\mu(\pi(s)\mid s),
  \qquad \pi\in\Theta.
\end{align*}
Then $\lambda_\pi\ge0$ and $\sum_{\pi\in\Theta}\lambda_\pi=1$. The matrix
$\Pi^\mu$ is the corresponding convex combination of the deterministic selector
matrices $\Pi^\pi$, and therefore \Cref{eq:policy-convex-hull} follows after
multiplication by $P$. The row-stochasticity and the identity
$P\Pi^\mu\mathbf{1}=\mathbf{1}$ follow from the stochasticity of $P$ and
$\mu$.
\end{proof}

\begin{lemma}
\label{lem:bellman_difference_fixed_point_selector}
For each $k$, let $\mu_k:=\mu_{Q_k,Q^\star}$ denote one stochastic
Bellman-difference selector. Then
\begin{equation*}
  F(Q_k)-F(Q^\star)=\gamma P\Pi^{\mu_k}(Q_k-Q^\star).
\end{equation*}
\end{lemma}

\subsection{Standard Constant/Projection Block Form}
\label{app:standard-constant-projection-block-form}

\begin{lemma}
\label{lem:standard-constant-projected-block-form}
Let $q:=\mathbf{1}/\sqrt n$ and $S:=[q\ U]$. For every deterministic policy
$\pi\in\Theta$,
\begin{align*}
  S^\top A_\pi^{\rm QVI}S
  =
  \begin{bmatrix}
    \bigl(1-\alpha(1-\gamma)\bigr)&q^\top A_\pi^{\rm QVI}U\\
    0&\bar A_\pi^{\rm QVI}
  \end{bmatrix}.
\end{align*}
Consequently, for every word $(\pi_0,\ldots,\pi_{k-1})$,
\begin{equation*}
\begin{aligned}
&S^\top A_{\pi_{k-1}}^{\rm QVI}\cdots A_{\pi_0}^{\rm QVI}S\\
&\quad=
  \begin{bmatrix}
    \bigl(1-\alpha(1-\gamma)\bigr)^k&*\\
    0&\bar A_{\pi_{k-1}}^{\rm QVI}\cdots\bar A_{\pi_0}^{\rm QVI}
  \end{bmatrix},
\end{aligned}
\end{equation*}
where $*$ denotes an unspecified block of conforming dimension.
\end{lemma}
\begin{proof}
Because $q=\mathbf{1}/\sqrt n$ and $P\Pi^\pi\mathbf{1}=\mathbf{1}$, every
standard QVI mode satisfies
\begin{align*}
  A_\pi^{\rm QVI}q=\bigl(1-\alpha(1-\gamma)\bigr)q.
\end{align*}
Together with $q^\top q=1$, $U^\top q=0$, and
$\bar A_\pi^{\rm QVI}=U^\top A_\pi^{\rm QVI}U$, this gives
\begin{align*}
  q^\top A_\pi^{\rm QVI}q&=\bigl(1-\alpha(1-\gamma)\bigr),\\
  U^\top A_\pi^{\rm QVI}q&=\bigl(1-\alpha(1-\gamma)\bigr)U^\top q=0,\\
  U^\top A_\pi^{\rm QVI}U&=\bar A_\pi^{\rm QVI}.
\end{align*}
Therefore
\begin{align*}
  S^\top A_\pi^{\rm QVI}S
  &=
  \begin{bmatrix}q^\top\\U^\top\end{bmatrix}
  A_\pi^{\rm QVI}
  \begin{bmatrix}q&U\end{bmatrix}\\
  &=
  \begin{bmatrix}
    q^\top A_\pi^{\rm QVI}q&q^\top A_\pi^{\rm QVI}U\\
    U^\top A_\pi^{\rm QVI}q&U^\top A_\pi^{\rm QVI}U
  \end{bmatrix}\\
  &=
  \begin{bmatrix}
    \bigl(1-\alpha(1-\gamma)\bigr)&q^\top A_\pi^{\rm QVI}U\\
    0&\bar A_\pi^{\rm QVI}
  \end{bmatrix}.
\end{align*}
Products of these transformed matrices remain block upper triangular. Their
upper-left diagonal block is multiplied by the same scalar at each step, and
their lower-right diagonal block is the corresponding product of the projected
modes. This proves the product identity.
\end{proof}

\subsection{Block Column Norm Bound}
\label{app:block-column-norm-bound}

\begin{lemma}
\label{lem:block-column-norm-bound}
For any conforming matrices $X$ and $Y$,
\begin{align*}
  \left\lVert
    \begin{bmatrix}X&0\\Y&0\end{bmatrix}
  \right\rVert_2
  \le \left\lVert X\right\rVert_2+\left\lVert Y\right\rVert_2.
\end{align*}
\end{lemma}
\begin{proof}
For any vector $\begin{bmatrix}u\\v\end{bmatrix}$,
\begin{align*}
  \left\lVert
  \begin{bmatrix}X&0\\Y&0\end{bmatrix}
  \begin{bmatrix}u\\v\end{bmatrix}
  \right\rVert_2
  &=
  \left\lVert
  \begin{bmatrix}Xu\\Yu\end{bmatrix}
  \right\rVert_2 \\
  &\le
  \left(\left\lVert X\right\rVert_2+\left\lVert Y\right\rVert_2\right)
  \left\lVert u\right\rVert_2 \\
  &\le
  \left(\left\lVert X\right\rVert_2+\left\lVert Y\right\rVert_2\right)
  \left\lVert\begin{bmatrix}u\\v\end{bmatrix}\right\rVert_2.
\end{align*}
Taking the supremum over all nonzero $\begin{bmatrix}u\\v\end{bmatrix}$ gives the claim.
\end{proof}

\subsection{Projected Product-Norm Bound}
\label{app:projected-product-norm-bound}

\begin{lemma}
\label{lem:projected-product-norm-bound}
Let $q:=\mathbf{1}/\sqrt n$ and $S:=[q\ U]$. For every deterministic word
$(\pi_0,\ldots,\pi_{k-1})$,
\begin{align*}
  \left\lVert
    \bar A_{\pi_{k-1}}^{\rm QVI}\cdots \bar A_{\pi_0}^{\rm QVI}
  \right\rVert_2
  &\le
  \left\lVert
    S^\top A_{\pi_{k-1}}^{\rm QVI}\cdots A_{\pi_0}^{\rm QVI}S
  \right\rVert_2\\
  &=
  \left\lVert
    A_{\pi_{k-1}}^{\rm QVI}\cdots A_{\pi_0}^{\rm QVI}
  \right\rVert_2.
\end{align*}
\end{lemma}
\begin{proof}
The block form gives
\begin{equation*}
\begin{aligned}
&S^\top A_{\pi_{k-1}}^{\rm QVI}\cdots A_{\pi_0}^{\rm QVI}S =
  \begin{bmatrix}
    \bigl(1-\alpha(1-\gamma)\bigr)^k&*\\
    0&\bar A_{\pi_{k-1}}^{\rm QVI}\cdots \bar A_{\pi_0}^{\rm QVI}
  \end{bmatrix},
\end{aligned}
\end{equation*}
where $*$ denotes an unspecified block of conforming dimension whose value is not
used below. Hence, for every $y\in\R^{n-1}$ with $y\ne0$,
\begin{align*}
  &S^\top A_{\pi_{k-1}}^{\rm QVI}\cdots A_{\pi_0}^{\rm QVI}S
  \begin{bmatrix}0\\y\end{bmatrix}
  =
  \begin{bmatrix}
    *\\
    \bar A_{\pi_{k-1}}^{\rm QVI}\cdots \bar A_{\pi_0}^{\rm QVI}y
  \end{bmatrix},
\end{align*}
and therefore
\begin{align*}
  \frac{
  \left\lVert
    \bar A_{\pi_{k-1}}^{\rm QVI}\cdots \bar A_{\pi_0}^{\rm QVI}y
  \right\rVert_2}{\left\lVert y\right\rVert_2}
  &\le
  \frac{
  \left\lVert
    S^\top A_{\pi_{k-1}}^{\rm QVI}\cdots A_{\pi_0}^{\rm QVI}S
    \begin{bmatrix}0\\y\end{bmatrix}
  \right\rVert_2}{\left\lVert \begin{bmatrix}0\\y\end{bmatrix}\right\rVert_2}\\
  &\le
  \left\lVert
    S^\top A_{\pi_{k-1}}^{\rm QVI}\cdots A_{\pi_0}^{\rm QVI}S
  \right\rVert_2.
\end{align*}
Taking the supremum over all nonzero $y$ gives
\begin{align*}
&\left\lVert
    \bar A_{\pi_{k-1}}^{\rm QVI}\cdots \bar A_{\pi_0}^{\rm QVI}
  \right\rVert_2\\
&\quad=
  \sup_{y\in\R^{n-1},\,y\ne0}
  \frac{\left\lVert
    \bar A_{\pi_{k-1}}^{\rm QVI}\cdots \bar A_{\pi_0}^{\rm QVI}y
  \right\rVert_2}{\left\lVert y\right\rVert_2}\\
&\quad\le
  \left\lVert
    S^\top A_{\pi_{k-1}}^{\rm QVI}\cdots A_{\pi_0}^{\rm QVI}S
  \right\rVert_2.
\end{align*}
Since $S$ is orthogonal, $S^\top S=SS^\top=I$ and orthogonal multiplication
preserves the Euclidean norm. Hence $\|S^\top v\|_2=\|v\|_2$ and
$\|Sx\|_2=\|x\|_2$, so
\begin{align*}
&\left\lVert
  S^\top A_{\pi_{k-1}}^{\rm QVI}\cdots A_{\pi_0}^{\rm QVI}S
  \right\rVert_2\\
&\quad=\sup_{x\ne0}
  \frac{\left\lVert
  S^\top A_{\pi_{k-1}}^{\rm QVI}\cdots A_{\pi_0}^{\rm QVI}Sx
  \right\rVert_2}{\left\lVert x\right\rVert_2}\\
&\quad=\sup_{x\ne0}
  \frac{\left\lVert
  A_{\pi_{k-1}}^{\rm QVI}\cdots A_{\pi_0}^{\rm QVI}Sx
  \right\rVert_2}{\left\lVert Sx\right\rVert_2}\\
&\quad=\sup_{y\ne0}
  \frac{\left\lVert
  A_{\pi_{k-1}}^{\rm QVI}\cdots A_{\pi_0}^{\rm QVI}y
  \right\rVert_2}{\left\lVert y\right\rVert_2}\\
&\quad=\left\lVert A_{\pi_{k-1}}^{\rm QVI}\cdots A_{\pi_0}^{\rm QVI}\right\rVert_2.
\end{align*}
\end{proof}

\subsection{JSR Continuity for the Projected Heavy-Ball Family}
\label{app:projected-hb-jsr-continuity}

The next auxiliary fact is the continuity property used in the global JSR
improvement argument.
\begin{lemma}
\label{lem:projected-hb-jsr-continuity}
The map
\begin{align*}
  \eta\mapsto \rho(\bar{\mathcal A}^{\rm HB})
\end{align*}
is continuous for the projected heavy-ball family in
\Cref{eq:projected-hb-family}.
\end{lemma}
\begin{proof}
The family in~\Cref{eq:projected-hb-family} is finite for each fixed $\eta$, and
each of its matrices depends continuously, in fact affinely, on $\eta$. The JSR
of a finite matrix family is continuous with respect to the entries of the
matrices in the family; see, for example,~\citep[Chapter~1]{jungers2009joint}.
Applying this standard continuity property to the projected heavy-ball family
gives the claim.
\end{proof}

\subsection{Common Eigenvector Invariance}
\label{app:common-eigenvector-invariance}

The appendix verifies the invariant augmented subspace associated with the common eigenvector used in the
main text.
\begin{lemma}
\label{lem:hb-common-eigenvector-invariance}
The two-dimensional augmented subspace $\mathcal S_{\mathbf{1}}^{\rm HB}$ defined in
\Cref{eq:hb-common-eigenvector-subspace} is invariant for every $A_\mu^{\rm HB}$
and therefore for every $A_\pi^{\rm HB}\in\mathcal A^{\rm HB}$. Equivalently, for every stochastic policy
$\mu$,
\begin{equation*}
  A_\mu^{\rm HB}\mathcal S_{\mathbf{1}}^{\rm HB}
  \subseteq \mathcal S_{\mathbf{1}}^{\rm HB}.
\end{equation*}
That is, applying any heavy-ball mode to a vector whose two blocks are constant
vectors again gives a vector whose two blocks are constant vectors.
\end{lemma}
\begin{proof}
For all $x,y\in\R$ and every stochastic policy $\mu$,
\begin{align*}
  A_\mu^{\rm HB}
  \begin{bmatrix}x\mathbf{1}\\y\mathbf{1}\end{bmatrix}
  &=
  \begin{bmatrix}
    \{(1-\alpha+\eta)I+\alpha\gamma P\Pi^\mu\}x\mathbf{1}-\eta y\mathbf{1}\\
    x\mathbf{1}
  \end{bmatrix}\\
  &=
  \begin{bmatrix}
    (1-\alpha+\eta)x\mathbf{1}+\alpha\gamma xP\Pi^\mu\mathbf{1}-\eta y\mathbf{1}\\
    x\mathbf{1}
  \end{bmatrix}\\
  &=
  \begin{bmatrix}
    \{\bigl(1-\alpha(1-\gamma)\bigr)+\eta\}x\mathbf{1}-\eta y\mathbf{1}\\
    x\mathbf{1}
  \end{bmatrix}
  \in\mathcal S_{\mathbf{1}}^{\rm HB},
\end{align*}
where the third equality uses $P\Pi^\mu\mathbf{1}=\mathbf{1}$. This proves
\Cref{eq:hb-common-eigenvector-invariance}.
\end{proof}

\subsection{Heavy-Ball JSR Upper Bound}
\label{app:hb-jsr-upper-bound}

\begin{proposition}
\label{prop:hb-jsr-upper-bound}
For $\eta\ge0$, we have
\begin{align*}
\rho(\mathcal A^{\rm HB})
&\le \rho(C_\eta^+)\\
&=\frac{1-\alpha(1-\gamma)+\eta}{2} + \frac{\sqrt{\bigl(1-\alpha(1-\gamma)+\eta\bigr)^2+4\eta}}{2},
\end{align*}
where
\begin{align*}
  C_\eta^+&:=
  \begin{bmatrix}
    \bigl(1-\alpha(1-\gamma)\bigr)+\eta&\eta\\
    1&0
  \end{bmatrix}.
\end{align*}
Consequently, the sufficient ambient stability condition
\begin{equation}
  0\le \eta<\frac{\alpha(1-\gamma)}{2}
  \label{eq:hb-simple-stability-sufficient}
\end{equation}
implies $\rho(\mathcal A^{\rm HB})<1$.
\end{proposition}

\subsection{Jury Stability Criterion}
\label{app:jury-stability-criterion}

The last auxiliary fact is the scalar quadratic stability test used for the
common-eigenvector companion polynomial.
\begin{lemma}
\label{lem:jury-quadratic}
For the real quadratic
\begin{align*}
  p(z)=z^2-az+b,
\end{align*}
all roots of $p$ belong to $\left\{z\in\mathbb C:|z|<1\right\}$ if and only if
\begin{equation*}
  1-a+b>0,
  \qquad
  1+a+b>0,
  \qquad
  1-b>0.
\end{equation*}
This is the second-order Jury stability criterion of~\citet{jury1962simplified}.
\end{lemma}

The proof is omitted; see~\citet{jury1962simplified}.

\section{Proofs of Main Results}
\label{app:proofs-main-results}

\subsection{Proof of Lemma~\ref{lem:convex-hull-jsr}}
\label{app:proof-lem-convex-hull-jsr}
\begin{proof}
This is the standard convex-hull invariance property of the JSR for finite matrix
families. Since every product whose factors are drawn from $\mathcal H$ is also a
product whose factors are drawn from $\co(\mathcal H)$, one has
$\rho(\mathcal H)\leq \rho(\co(\mathcal H))$. Conversely, for each fixed product
length, multilinearity of matrix products and convexity of norms imply that the
largest norm among products with factors in $\co(\mathcal H)$ is attained no
larger than the largest norm among products with factors in $\mathcal H$. Taking
$k$th roots and then the limit gives the reverse inequality. This is the usual
convex-hull invariance argument for the JSR~\citep[Chapter~1]{jungers2009joint}.
\end{proof}

\subsection{Proof of Lemma~\ref{lem:block-triangular-jsr}}
\label{app:proof-lem-block-triangular-jsr}

\begin{proof}
This is the block-triangular decomposition property of the JSR; see \citep[Proposition~1.5]{jungers2009joint}. Products of matrices in
$\mathcal T$ remain block upper triangular, and their diagonal blocks are exactly
the corresponding products of the $B_i$ blocks and the $D_i$ blocks. The
extra-diagonal blocks do not change the exponential growth rate that defines the
JSR, so the JSR is the maximum of the JSRs of the diagonal block families.
\end{proof}

\subsection{Proof of Lemma~\ref{lem:common_lyapunov_construction}}
\label{app:proof-lem-common-lyapunov-construction}

\begin{proof}
The construction and proof are given by~\citet{lee2026lyapunovcertified,hushenzhang2010generating}; we
omit the proof here.
\end{proof}

\subsection{Proof of Lemma~\ref{lem:standard_jsr_convergence}}
\label{app:proof-lem-standard-jsr-convergence}

\begin{proof}
By~\Cref{lem:common_lyapunov_construction}, $p_\epsilon(A_i x)\leq\beta_\epsilon p_\epsilon(x)$ for every $A_i\in\mathcal H$. If
$A_k=\sum_i\lambda_{k,i}A_i$ is a convex combination, then the triangle inequality
and homogeneity of the norm give $p_\epsilon(A_kx)
\leq
\sum_i\lambda_{k,i}p_\epsilon(A_i x)
\leq
\beta_\epsilon p_\epsilon(x)$. Applying this with $x=x_k$ gives the one-step norm contraction, and squaring it
gives the stated decrease of $V_\epsilon^\infty$. Iteration gives the
$p_\epsilon$ estimate, and the Euclidean bound follows from
$\|x\|_2\leq p_\epsilon(x)\leq\sqrt{C_\epsilon}\|x\|_2$.
\end{proof}

\subsection{Proof of Lemma~\ref{lem:bellman_difference_selector}}
\label{app:proof-lem-bellman-difference-selector}

\begin{proof}
The state-wise convex-interpolation argument is given in \Cref{lem:bellman_difference_selector} in Appendix.
\end{proof}

\subsection{Proof of Lemma~\ref{lem:bellman_difference_fixed_point_selector}}
\label{app:proof-lem-bellman-difference-fixed-point-selector}

\begin{proof}
Apply~\Cref{lem:bellman_difference_selector} with $Q=Q_k$ and
$\bar Q=Q^\star$.
\end{proof}

\subsection{Proof of Lemma~\ref{lem:standard-q-error-system}}
\label{app:proof-lem-standard-q-error-system}

\begin{proof}
Subtracting $Q^\star=F(Q^\star)$ from~\Cref{eq:standard-q-update} gives
\begin{align*}
  Q_{k+1}-Q^\star
  &=(1-\alpha)(Q_k-Q^\star)
    +\alpha\{F(Q_k)-F(Q^\star)\}.
\end{align*}
By~\Cref{lem:bellman_difference_fixed_point_selector}, we have $F(Q_k)-F(Q^\star) =\gamma P\Pi^{\mu_k}(Q_k-Q^\star)$. Therefore, we obtain the SLS representation
$Q_{k+1}-Q^\star =\{(1-\alpha)I+\alpha\gamma P\Pi^{\mu_k}\}(Q_k-Q^\star)=A_{\mu_k}^{\rm QVI}(Q_k-Q^\star)$, which completes the proof.
\end{proof}

\subsection{Proof of Proposition~\ref{prop:standard_q_jsr}}
\label{app:proof-prop-standard-q-jsr}

\begin{proof}
For every deterministic policy $\pi$, the matrix $P\Pi^\pi$ is row-stochastic.
Thus $A_\pi^{\rm QVI}$ is nonnegative under $0<\alpha\le1$, and every row of
$A_\pi^{\rm QVI}$ sums to
\begin{align*}
  (1-\alpha)+\alpha\gamma=1-\alpha(1-\gamma).
\end{align*}
Therefore
\begin{align*}
  \left\lVert A_\pi^{\rm QVI}\right\rVert_\infty
  =1-\alpha(1-\gamma),
  \qquad \pi\in\Theta.
\end{align*}
Moreover, for any word $(\pi_0,\ldots,\pi_{k-1})$, the product
$A_{\pi_{k-1}}^{\rm QVI}\cdots A_{\pi_0}^{\rm QVI}$ is again nonnegative and its
rows sum to $\bigl(1-\alpha(1-\gamma)\bigr)^k$. Hence
\begin{align*}
  \left\lVert A_{\pi_{k-1}}^{\rm QVI}\cdots A_{\pi_0}^{\rm QVI}\right\rVert_\infty
  =\bigl(1-\alpha(1-\gamma)\bigr)^k.
\end{align*}
Taking the maximum over all words of length $k$, taking $k$th roots, and passing
to the JSR limit gives
\begin{align*}
  \rho(\mathcal A^{\rm QVI})=1-\alpha(1-\gamma).
\end{align*}
The convex-hull equality follows from~\Cref{lem:convex-hull-jsr}.
\end{proof}

\subsection{Proof of Lemma~\ref{lem:standard-qvi-common-eigenvector}}
\label{app:proof-lem-standard-qvi-common-eigenvector}

\begin{proof}
Since $P\Pi^\mu\mathbf{1}=\mathbf{1}$ for every stochastic policy $\mu$, it follows that
\begin{align*}
  A_\mu^{\rm QVI}\mathbf{1}
  &=\bigl((1-\alpha)I+\alpha\gamma P\Pi^\mu\bigr)\mathbf{1}\\
  &=\bigl(1-\alpha+\alpha\gamma\bigr)\mathbf{1}
  =\bigl(1-\alpha(1-\gamma)\bigr)\mathbf{1}.
\end{align*}
Therefore, the common eigenvector direction is invariant for every standard QVI mode. If
$Q_k-Q^\star=a_k\mathbf{1}$, then the standard error recursion in
\Cref{lem:standard-q-error-system} gives
\begin{align*}
  Q_{k+1}-Q^\star
  &=A_{\mu_k}^{\rm QVI}(Q_k-Q^\star)\\
  &=a_k A_{\mu_k}^{\rm QVI}\mathbf{1}
  =\bigl(1-\alpha(1-\gamma)\bigr)a_k\mathbf{1},
\end{align*}
which is the stated scalar update.
\end{proof}

\subsection{Proof of Lemma~\ref{lem:projected-error-systems}}
\label{app:proof-lem-projected-error-systems}

\begin{proof}
Since $A_{\mu_k}^{\rm QVI}\mathbf{1}=(1-\alpha(1-\gamma))\mathbf{1}$ and
$U^\top\mathbf{1}=0$,
\begin{align*}
  U^\top A_{\mu_k}^{\rm QVI}(I-UU^\top)=0.
\end{align*}
Subtracting the fixed point $Q^\star=F(Q^\star)$ and projecting the affine
Bellman selector recursion therefore gives
\begin{align*}
  U^\top(Q_{k+1}-Q^\star)
  &=U^\top A_{\mu_k}^{\rm QVI}(Q_k-Q^\star)\\
  &=U^\top A_{\mu_k}^{\rm QVI}UU^\top(Q_k-Q^\star)\\
  &=\bar A_{\mu_k}^{\rm QVI}z_k.
\end{align*}
\end{proof}

\subsection{Proof of Lemma~\ref{lem:qvi-common-eigenvector-projected-jsr-decomposition}}
\label{app:proof-lem-qvi-common-eigenvector-projected-jsr-decomposition}

\begin{proof}
Let $q:=\mathbf{1}/\sqrt n$ and $S:=[q\ U]$. By
\Cref{lem:standard-constant-projected-block-form}, for every deterministic
policy $\pi$,
\begin{align*}
  S^\top A_\pi^{\rm QVI}S
  =
  \begin{bmatrix}
    1-\alpha(1-\gamma)&q^\top A_\pi^{\rm QVI}U\\
    0&\bar A_\pi^{\rm QVI}
  \end{bmatrix}.
\end{align*}
Products of these transformed matrices remain block upper triangular. For every
word $(\pi_0,\ldots,\pi_{k-1})$, the two diagonal blocks are
\begin{align*}
  \bigl(1-\alpha(1-\gamma)\bigr)^k
  \quad\text{and}\quad
  \bar A_{\pi_{k-1}}^{\rm QVI}\cdots\bar A_{\pi_0}^{\rm QVI}.
\end{align*}
The block triangular JSR decomposition in~\Cref{lem:block-triangular-jsr}
therefore gives the maximum of the scalar-block JSR and the projected-block JSR,
which proves the claim.
\end{proof}

\subsection{Proof of Lemma~\ref{lem:projected-ql-weak-bound}}
\label{app:proof-lem-projected-ql-weak-bound}

\begin{proof}
By~\Cref{lem:qvi-common-eigenvector-projected-jsr-decomposition},
\begin{align*}
  \rho(\mathcal A^{\rm QVI})
  =\max\left\{1-\alpha(1-\gamma),\rho(\bar{\mathcal A}^{\rm QVI})\right\}.
\end{align*}
Hence $\rho(\bar{\mathcal A}^{\rm QVI})\le\rho(\mathcal A^{\rm QVI})$.
The equality
$\rho(\mathcal A^{\rm QVI})=1-\alpha(1-\gamma)$ follows from
\Cref{prop:standard_q_jsr}.
\end{proof}

\subsection{Proof of Lemma~\ref{lem:hb-fixed-point-unchanged}}
\label{app:proof-lem-hb-fixed-point-unchanged}

\begin{proof}
By definition of $g$, the heavy-ball recursion is
\begin{align*}
  \begin{bmatrix}Q_{k+1}\\Q_k\end{bmatrix}=g(Q_k,Q_{k-1}).
\end{align*}
Let an augmented fixed point be denoted by $\begin{bmatrix}\bar Q\\\tilde Q\end{bmatrix}$. Then, it must satisfy
\begin{align*}
  \begin{bmatrix}\bar Q\\\tilde Q\end{bmatrix}
  =g(\bar Q,\tilde Q)
  =
  \begin{bmatrix}
    (1-\alpha+\eta)\bar Q-\eta\tilde Q+\alpha F(\bar Q)\\
    \bar Q
  \end{bmatrix}.
\end{align*}
The second block gives $\tilde Q=\bar Q$. Substituting this identity into the
first block yields
\begin{align*}
  \bar Q
  &=(1-\alpha+\eta)\bar Q-\eta\bar Q+\alpha F(\bar Q)\\
  &=(1-\alpha)\bar Q+\alpha F(\bar Q).
\end{align*}
Hence, $\alpha\{F(\bar Q)-\bar Q\}=0$. Since $\alpha>0$, $F(\bar Q)=\bar Q$. The discounted Bellman optimality operator
has the unique fixed point $Q^\star$, so $\bar Q=Q^\star$ and therefore also
$\tilde Q=Q^\star$. Conversely, because $F(Q^\star)=Q^\star$,
\begin{align*}
g(Q^\star,Q^\star)
  &=
  \begin{bmatrix}
    (1-\alpha+\eta)Q^\star-\eta Q^\star+\alpha F(Q^\star)\\
    Q^\star
  \end{bmatrix}\\
  &=
  \begin{bmatrix}Q^\star\\Q^\star\end{bmatrix}.
\end{align*}
Therefore, $(Q^\star,Q^\star)$ is indeed an augmented fixed point.
\end{proof}

\subsection{Proof of Lemma~\ref{lem:hb-exact-sls}}
\label{app:proof-lem-hb-exact-sls}

\begin{proof}
Subtracting $Q^\star=F(Q^\star)$ from~\Cref{eq:simple-heavy-ball} gives
\begin{align*}
Q_{k+1}-Q^\star
  &=(1-\alpha+\eta)(Q_k-Q^\star)
    -\eta(Q_{k-1}-Q^\star)\\
  &\quad+\alpha\{F(Q_k)-F(Q^\star)\}.
\end{align*}
Using~\Cref{lem:bellman_difference_fixed_point_selector}, we have
\begin{align*}
  F(Q_k)-F(Q^\star)=\gamma P\Pi^{\mu_k}(Q_k-Q^\star),
\end{align*}
and hence
\begin{align*}
Q_{k+1}-Q^\star
  &=\left[(1-\alpha+\eta)I+\alpha\gamma P\Pi^{\mu_k}\right]
    (Q_k-Q^\star)\\
  &\quad-\eta(Q_{k-1}-Q^\star).
\end{align*}
Stacking this identity with the lag identity
$Q_k-Q^\star=I(Q_k-Q^\star)+0\cdot(Q_{k-1}-Q^\star)$ gives
\Cref{eq:simple-error}. The convex-hull statement follows from the linearity of
$A_\mu^{\rm HB}$ in $P\Pi^\mu$ and from~\Cref{eq:policy-convex-hull}.
\end{proof}

\subsection{Proof of Lemma~\ref{lem:hb-common-eigenvector-recursion}}
\label{app:proof-lem-hb-common-eigenvector-recursion}

\begin{proof}
Taking $x=a_k$ and $y=a_{k-1}$ in the invariant-subspace calculation above gives
\begin{align*}
  Q_{k+1}-Q^\star
  &=\{\bigl(1-\alpha(1-\gamma)\bigr)+\eta\}a_k\mathbf{1}
    -\eta a_{k-1}\mathbf{1},\\
  Q_k-Q^\star&=a_k\mathbf{1}.
\end{align*}
Thus $Q_{k+1}-Q^\star=a_{k+1}\mathbf{1}$ with $a_{k+1}$ given by
\Cref{eq:common-eigenvector-scalar-recursion}. Stacking the two scalar coordinates gives the companion recursion.
\end{proof}

\subsection{Proof of Lemma~\ref{lem:hb-common-eigenvector-spectral-replacement}}
\label{app:proof-lem-hb-common-eigenvector-spectral-replacement}

\begin{proof}
For every deterministic policy $\pi\in\Theta$,
\begin{align*}
  A_\pi^{\rm QVI}\mathbf{1}
  &=\bigl(1-\alpha(1-\gamma)\bigr)\mathbf{1},\\
  A_\pi^{\rm HB}
  \begin{bmatrix}x\mathbf{1}\\y\mathbf{1}\end{bmatrix}
  &=
  \begin{bmatrix}
    \{\bigl(1-\alpha(1-\gamma)\bigr)+\eta\}x\mathbf{1}-\eta y\mathbf{1}\\
    x\mathbf{1}
  \end{bmatrix}.
\end{align*}
Thus the heavy-ball common-eigenvector coordinates $(x,y)$ are multiplied by
$C_\eta^-$, and
\begin{align*}
  \det(\lambda I-C_\eta^-)
  &=\det
  \begin{bmatrix}
    \lambda-\bigl(1-\alpha(1-\gamma)\bigr)-\eta&\eta\\
    -1&\lambda
  \end{bmatrix}\\
  &=\lambda^2-\bigl(\bigl(1-\alpha(1-\gamma)\bigr)+\eta\bigr)\lambda+
  \eta.
\end{align*}
For every word $(\pi_0,\ldots,\pi_{k-1})$, we can write
\begin{align*}
  A_{\pi_{k-1}}^{\rm HB}\cdots A_{\pi_0}^{\rm HB}
  \begin{bmatrix}x\mathbf{1}\\y\mathbf{1}\end{bmatrix}
  =
  \begin{bmatrix}x_k\mathbf{1}\\y_k\mathbf{1}\end{bmatrix},
  \qquad
  \begin{bmatrix}x_k\\y_k\end{bmatrix}=(C_\eta^-)^k\begin{bmatrix}x\\y\end{bmatrix}.
\end{align*}
Therefore the product growth of $\mathcal A^{\rm HB}$ is at least the product
growth of $(C_\eta^-)^k$, and
\begin{align*}
  \rho(\mathcal A^{\rm HB})
  &\ge
  \limsup_{k\to\infty}\rho((C_\eta^-)^k)^{1/k}
  =\rho(C_\eta^-).
\end{align*}
This completes the proof.
\end{proof}

\subsection{Proof of Proposition~\ref{prop:common_eigenvector_acceleration}}
\label{app:proof-prop-common-eigenvector-acceleration}

\begin{proof}
By~\Cref{prop:standard_q_jsr}, the characteristic polynomial in~\Cref{eq:hb-common-eigenvector-poly} can be written as
\begin{align}
  \lambda^2-\bigl(\rho(\mathcal A^{\rm QVI})+\eta\bigr)\lambda+\eta=0.\label{eq:hb-scaled-characteristic}
\end{align}
The inequality $\rho(C_\eta^-)<\rho(\mathcal A^{\rm QVI})$ is equivalent to the
condition that both roots of this polynomial satisfy
$|\lambda|<\rho(\mathcal A^{\rm QVI})$. With
$\lambda=\rho(\mathcal A^{\rm QVI})z$,~\Cref{eq:hb-scaled-characteristic} can be written as
\begin{align*}
0=z^2- \left(1+\frac{\eta}{\rho(\mathcal A^{\rm QVI})}\right)z
    +\frac{\eta}{\rho(\mathcal A^{\rm QVI})^2}.
\end{align*}
We now derive a condition under which $|z|<1$.
To this end, we apply~\Cref{lem:jury-quadratic} with
\begin{align*}
  a=1+\frac{\eta}{\rho(\mathcal A^{\rm QVI})},
  \qquad
  b=\frac{\eta}{\rho(\mathcal A^{\rm QVI})^2}.
\end{align*}
In particular, the three Jury inequalities from~\Cref{lem:jury-quadratic} are
\begin{align*}
  1-a+b
  &=1-
    \left(1+\frac{\eta}{\rho(\mathcal A^{\rm QVI})}\right)
    +\frac{\eta}{\rho(\mathcal A^{\rm QVI})^2}\\
  &=\frac{\eta\bigl(1-\rho(\mathcal A^{\rm QVI})\bigr)}
    {\rho(\mathcal A^{\rm QVI})^2}>0,\\
  1+a+b
  &=1+
    \left(1+\frac{\eta}{\rho(\mathcal A^{\rm QVI})}\right)
    +\frac{\eta}{\rho(\mathcal A^{\rm QVI})^2}\\
  &=2+\frac{\eta}{\rho(\mathcal A^{\rm QVI})}
    +\frac{\eta}{\rho(\mathcal A^{\rm QVI})^2}>0,\\
  1-b
  &=1-\frac{\eta}{\rho(\mathcal A^{\rm QVI})^2}>0.
\end{align*}
Under the standing discounted and relaxation assumptions, these conditions are equivalent to
\begin{align*}
  \eta>0,
  \qquad
  \eta<\rho(\mathcal A^{\rm QVI})^2.
\end{align*}
Therefore
\begin{align*}
  \rho(C_\eta^-)<\rho(\mathcal A^{\rm QVI})
  &\Longleftrightarrow
  0<\eta<\rho(\mathcal A^{\rm QVI})^2.
\end{align*}
Finally,
\begin{align*}
\rho(\mathcal A^{\rm HB})&<\rho(\mathcal A^{\rm QVI})\\
&\Longrightarrow
  \rho(C_\eta^-)\le\rho(\mathcal A^{\rm HB})<\rho(\mathcal A^{\rm QVI}),
\end{align*}
by~\Cref{lem:hb-common-eigenvector-spectral-replacement}. Hence the strict
standard-family improvement requires $0<\eta<\rho(\mathcal A^{\rm QVI})^2$.
\end{proof}

\subsection{Proof of Proposition~\ref{prop:hb-jsr-upper-bound}}
\label{app:proof-prop-hb-jsr-upper-bound}

\begin{proof}
Let
\begin{align*}
  \begin{bmatrix}x_{j+1}\\x_j\end{bmatrix}
  =A_{\mu_j}^{\rm HB}\begin{bmatrix}x_j\\x_{j-1}\end{bmatrix},
  \qquad
  \mu_j\text{ is any stochastic policy}.
\end{align*}
For every $j\ge0$,
\begin{align*}
x_{j+1}
  &=\{(1-\alpha+\eta)I+\alpha\gamma P\Pi^{\mu_j}\}x_j
    -\eta x_{j-1},\\
\|x_{j+1}\|_\infty
  &\le(1-\alpha+\eta)\|x_j\|_\infty
     +\alpha\gamma\|P\Pi^{\mu_j}x_j\|_\infty\\
  &\quad+
     \eta\|x_{j-1}\|_\infty\\
  &\le(1-\alpha+\eta+\alpha\gamma)\|x_j\|_\infty
     +\eta\|x_{j-1}\|_\infty\\
  &=\{1-\alpha(1-\gamma)+\eta\}\|x_j\|_\infty
     +\eta\|x_{j-1}\|_\infty,\\
\|x_j\|_\infty&=\|x_j\|_\infty.
\end{align*}
Hence, componentwise in $\R^2_+$,
\begin{align*}
  \begin{bmatrix}\|x_{j+1}\|_\infty\\\|x_j\|_\infty\end{bmatrix}
  \le
  C_\eta^+
  \begin{bmatrix}\|x_j\|_\infty\\\|x_{j-1}\|_\infty\end{bmatrix},
\end{align*}
and induction gives
\begin{align*}
  \begin{bmatrix}\|x_k\|_\infty\\\|x_{k-1}\|_\infty\end{bmatrix}
  \le
  (C_\eta^+)^k
  \begin{bmatrix}\|x_0\|_\infty\\\|x_{-1}\|_\infty\end{bmatrix}.
\end{align*}
Now fix an arbitrary initial augmented vector
\begin{align*}
  w_0:=\begin{bmatrix}x_0\\x_{-1}\end{bmatrix}
\end{align*}
and let
\begin{align*}
  w_k:=A_{\mu_{k-1}}^{\rm HB}\cdots A_{\mu_0}^{\rm HB}w_0
  =\begin{bmatrix}x_k\\x_{k-1}\end{bmatrix}.
\end{align*}
The preceding componentwise inequality gives
\begin{align*}
  \begin{bmatrix}\|x_k\|_\infty\\\|x_{k-1}\|_\infty\end{bmatrix}
  &\le
  (C_\eta^+)^k
  \begin{bmatrix}\|x_0\|_\infty\\\|x_{-1}\|_\infty\end{bmatrix}.
\end{align*}
Since $C_\eta^+\ge0$, also $(C_\eta^+)^k\ge0$. Hence, for each row $i=1,2$,
\begin{align*}
&\left[(C_\eta^+)^k
  \begin{bmatrix}\|x_0\|_\infty\\\|x_{-1}\|_\infty\end{bmatrix}\right]_i\\
&\quad=
  \sum_{j=1}^2\left[(C_\eta^+)^k\right]_{ij}
  \begin{bmatrix}\|x_0\|_\infty\\\|x_{-1}\|_\infty\end{bmatrix}_j\\
&\quad\le
  \sum_{j=1}^2\left[(C_\eta^+)^k\right]_{ij}
  \bigl(\|x_0\|_\infty+\|x_{-1}\|_\infty\bigr)\\
&\quad\le
  2\left\|(C_\eta^+)^k\right\|_\infty
  \left\|\begin{bmatrix}x_0\\x_{-1}\end{bmatrix}\right\|_\infty.
\end{align*}
Therefore
\begin{align*}
  \left\|A_{\mu_{k-1}}^{\rm HB}\cdots A_{\mu_0}^{\rm HB}w_0\right\|_\infty
  &=\left\|\begin{bmatrix}x_k\\x_{k-1}\end{bmatrix}\right\|_\infty\\
  &=\max\{\|x_k\|_\infty,\|x_{k-1}\|_\infty\}\\
  &\le
  2\left\|(C_\eta^+)^k\right\|_\infty
  \left\|w_0\right\|_\infty.
\end{align*}
Taking the supremum over all $w_0\ne0$ gives, for any product
$A_{\mu_{k-1}}^{\rm HB}\cdots A_{\mu_0}^{\rm HB}$,
\begin{align*}
  \left\|A_{\mu_{k-1}}^{\rm HB}\cdots A_{\mu_0}^{\rm HB}\right\|_\infty
  \le
  2\left\|(C_\eta^+)^k\right\|_\infty.
\end{align*}
Therefore
\begin{align*}
  \rho(\mathcal A^{\rm HB})
  &\le
  \lim_{k\to\infty}
  \left(2\left\|(C_\eta^+)^k\right\|_\infty\right)^{1/k}
  =\rho(C_\eta^+).
\end{align*}
Moreover,
\begin{align*}
  \det(\lambda I-C_\eta^+)
  &=
  \det\begin{bmatrix}
    \lambda-\bigl(1-\alpha(1-\gamma)\bigr)-\eta&-\eta\\
    -1&\lambda
  \end{bmatrix}\\
  &=\lambda^2-\bigl(\bigl(1-\alpha(1-\gamma)\bigr)+\eta\bigr)\lambda-
  \eta,
\end{align*}
so, since $C_\eta^+\ge0$,
\begin{align*}
\rho(C_\eta^+)
  &=\frac{1-\alpha(1-\gamma)+\eta}{2}\\
  &\quad+
  \frac{\sqrt{\bigl(1-\alpha(1-\gamma)+\eta\bigr)^2+4\eta}}{2}.
\end{align*}
Finally,
\begin{align*}
\rho(C_\eta^+)<1
  &\Longleftrightarrow
  1-\alpha(1-\gamma)+\eta\\
  &\quad+
  \sqrt{\bigl(1-\alpha(1-\gamma)+\eta\bigr)^2+4\eta}<2\\
  &\Longleftrightarrow
  \sqrt{\bigl(1-\alpha(1-\gamma)+\eta\bigr)^2+4\eta}\\
  &\quad<1+\alpha(1-\gamma)-\eta\\
  &\Longleftrightarrow
  \bigl(1-\alpha(1-\gamma)+\eta\bigr)^2+4\eta\\
  &\quad<\bigl(1+\alpha(1-\gamma)-\eta\bigr)^2\\
  &\Longleftrightarrow 2\eta<\alpha(1-\gamma),
\end{align*}
and~\Cref{eq:hb-simple-stability-sufficient} gives
$\rho(\mathcal A^{\rm HB})\le\rho(C_\eta^+)<1$.
\end{proof}

\subsection{Proof of Lemma~\ref{lem:projected-hb-error-system}}
\label{app:proof-lem-projected-hb-error-system}

\begin{proof}
Write the unprojected heavy-ball error as
\begin{align*}
  Q_{k+1}-Q^\star
  &=\left(A_{\mu_k}^{\rm QVI}+\eta I\right)(Q_k-Q^\star)
    -\eta(Q_{k-1}-Q^\star).
\end{align*}
Applying $U^\top$, using the standard projection identity from
\Cref{lem:projected-error-systems}, and using $z_j=U^\top(Q_j-Q^\star)$ gives
\begin{align*}
  z_{k+1}
  &=\bar A_{\mu_k}^{\rm QVI}z_k+\eta z_k-\eta z_{k-1}\\
  &=\left(\bar A_{\mu_k}^{\rm QVI}+\eta I_{n-1}\right)z_k
    -\eta z_{k-1}.
\end{align*}
Stacking this identity with $z_k=I_{n-1}z_k+0\cdot z_{k-1}$ gives the
projected heavy-ball system.
\end{proof}

\subsection{Proof of Lemma~\ref{lem:hb-common-eigenvector-projected-jsr-decomposition}}
\label{app:proof-lem-hb-common-eigenvector-projected-jsr-decomposition}

\begin{proof}
Let
\begin{align*}
  q:=\frac{\mathbf{1}}{\sqrt{n}},
  \qquad
  S:=[q\ U],
\end{align*}
where $U$ is the matrix fixed in~\Cref{eq:orthogonal-projection,eq:projected-ql-family}; equivalently,
\begin{align*}
  U^\top U=I_{n-1},
  \qquad
  U^\top q=0,
  \qquad
  UU^\top=I-qq^\top.
\end{align*}
By~\Cref{lem:standard-constant-projected-block-form}, for every deterministic
policy $\pi$,
\begin{align*}
  S^\top A_\pi^{\rm QVI}S
  =
  \begin{bmatrix}
    \bigl(1-\alpha(1-\gamma)\bigr)&q^\top A_\pi^{\rm QVI}U\\
    0&\bar A_\pi^{\rm QVI}
  \end{bmatrix}.
\end{align*}
Since
\begin{align*}
  A_\pi^{\rm HB}
  &=
  \begin{bmatrix}
    A_\pi^{\rm QVI}+\eta I&-\eta I\\
    I&0
  \end{bmatrix},
\end{align*}
one obtains the exact augmented coordinate identity
\begin{equation*}
\begin{aligned}
&\begin{bmatrix}
    q^\top&0\\
    0&q^\top\\
    U^\top&0\\
    0&U^\top
  \end{bmatrix}
  A_\pi^{\rm HB}
  \begin{bmatrix}
    q&0&U&0\\
    0&q&0&U
  \end{bmatrix}\\
  &\quad=
  {\setlength{\arraycolsep}{3pt}
  \begin{bmatrix}
    1-\alpha(1-\gamma)+\eta&-\eta&q^\top A_\pi^{\rm QVI}U&0\\
    1&0&0&0\\
    0&0&\bar A_\pi^{\rm QVI}+\eta I_{n-1}
      &-\eta I_{n-1}\\
    0&0&I_{n-1}&0
  \end{bmatrix}}\\
  &\quad=
  \begin{bmatrix}
    C_\eta^-&B_\pi\\
    0&\bar A_\pi^{\rm HB}
  \end{bmatrix}.
\end{aligned}
\end{equation*}
Here
\begin{equation*}
  B_\pi:=
  \begin{bmatrix}
    q^\top A_\pi^{\rm QVI}U&0\\
    0&0
  \end{bmatrix}.
\end{equation*}
Consequently, for every sequence of policies $(\pi_0,\ldots,\pi_{k-1})$, we have
\begin{equation*}
\begin{aligned}
&\begin{bmatrix}
    C_\eta^-&B_{\pi_{k-1}}\\
    0&\bar A_{\pi_{k-1}}^{\rm HB}
  \end{bmatrix}
  \times \cdots \times
  \begin{bmatrix}
    C_\eta^-&B_{\pi_0}\\
    0&\bar A_{\pi_0}^{\rm HB}
  \end{bmatrix} =  \begin{bmatrix}
    (C_\eta^-)^k&*\\
    0&\bar A_{\pi_{k-1}}^{\rm HB}\cdots
      \bar A_{\pi_0}^{\rm HB}
  \end{bmatrix}.
\end{aligned}
\end{equation*}
Here $*$ denotes an unspecified block of conforming dimension whose value is not
used in the JSR computation.
By the block triangular JSR decomposition in~\Cref{lem:block-triangular-jsr},
\begin{align*}
  \rho(\mathcal A^{\rm HB})=
  \max\left\{\rho(C_\eta^-),\rho(\bar{\mathcal A}^{\rm HB})\right\}.
\end{align*}
By~\Cref{lem:convex-hull-jsr}, the same equality also holds with
$\mathcal A^{\rm HB}$ replaced by $\co(\mathcal A^{\rm HB})$.
\end{proof}

\subsection{Proof of Lemma~\ref{lem:projected-hb-weak-bound}}
\label{app:proof-lem-projected-hb-weak-bound}

\begin{proof}
This follows immediately from the exact decomposition in
\Cref{lem:hb-common-eigenvector-projected-jsr-decomposition}: $\rho(\mathcal A^{\rm HB})
  =\max\left\{\rho(C_\eta^-),\rho(\bar{\mathcal A}^{\rm HB})\right\}$.
\end{proof}

\subsection{Proof of Lemma~\ref{lem:full-hb-zero-momentum}}
\label{app:proof-lem-full-hb-zero-momentum}

\begin{proof}
When $\eta=0$, every full heavy-ball matrix has the form
\begin{align*}
  A_\pi^{\rm HB}
  =
  \begin{bmatrix}
    A_\pi^{\rm QVI}&0\\
    I&0
  \end{bmatrix}.
\end{align*}
For every sequence of deterministic policies $(\pi_0,\ldots,\pi_{k-1})$ with $k\ge2$, we obtain
\begin{align}
  A_{\pi_{k-1}}^{\rm HB}\cdots A_{\pi_0}^{\rm HB}
  =
  \begin{bmatrix}
    A_{\pi_{k-1}}^{\rm QVI}\cdots A_{\pi_0}^{\rm QVI}&0\\
    A_{\pi_{k-2}}^{\rm QVI}\cdots A_{\pi_0}^{\rm QVI}&0
  \end{bmatrix}.
  \label{eq:zero-momentum-full-product}
\end{align}
The upper-left block gives the lower bound in the JSR sense. Indeed, using the
block projection onto the first component, for every such word we have
\begin{align*}
  \left\lVert
    A_{\pi_{k-1}}^{\rm QVI}\cdots A_{\pi_0}^{\rm QVI}
  \right\rVert_2
  &=
  \left\lVert
  \begin{bmatrix}I&0\end{bmatrix}
  A_{\pi_{k-1}}^{\rm HB}\cdots A_{\pi_0}^{\rm HB}
  \begin{bmatrix}I\\0\end{bmatrix}
  \right\rVert_2 \\
  &\le
  \left\lVert
    A_{\pi_{k-1}}^{\rm HB}\cdots A_{\pi_0}^{\rm HB}
  \right\rVert_2.
\end{align*}
After maximizing over all words of length $k$, taking $k$th roots, and passing
to the JSR limit in~\Cref{def:jsr}, this gives
$\rho(\mathcal A^{\rm QVI})\le\rho(\mathcal A^{\rm HB})$. Conversely, the block product in~\Cref{eq:zero-momentum-full-product} contains only two nonzero blocks. By the block-column norm bound in~\Cref{lem:block-column-norm-bound}, applying this inequality to the zero-momentum product gives, for every $k\ge2$,
\begin{align*}
&\left\lVert
    A_{\pi_{k-1}}^{\rm HB}\cdots A_{\pi_0}^{\rm HB}
  \right\rVert_2 \le
  \left\lVert
    A_{\pi_{k-1}}^{\rm QVI}\cdots A_{\pi_0}^{\rm QVI}
  \right\rVert_2
  +
  \left\lVert
    A_{\pi_{k-2}}^{\rm QVI}\cdots A_{\pi_0}^{\rm QVI}
  \right\rVert_2.
\end{align*}
Therefore, we obtain
\begin{align*}
&\max_{\pi_0,\ldots,\pi_{k-1}}
  \left\|A_{\pi_{k-1}}^{\rm HB}\cdots A_{\pi_0}^{\rm HB}\right\|_2 \\
&\quad\le
  \max_{\pi_0,\ldots,\pi_{k-1}}
  \left\|A_{\pi_{k-1}}^{\rm QVI}\cdots A_{\pi_0}^{\rm QVI}\right\|_2
  +
  \max_{\pi_0,\ldots,\pi_{k-2}}
  \left\|A_{\pi_{k-2}}^{\rm QVI}\cdots A_{\pi_0}^{\rm QVI}\right\|_2.
\end{align*}
The two terms on the right do introduce a factor of $2$ after we bound their
sum by twice their maximum. More explicitly,
\begin{align*}
&\left(
  \max_{\pi_0,\ldots,\pi_{k-1}}
  \left\|A_{\pi_{k-1}}^{\rm HB}\cdots A_{\pi_0}^{\rm HB}\right\|_2
  \right)^{1/k} \\
&\quad\le 2^{1/k}\max\left\{
  \left(
  \max_{\pi_0,\ldots,\pi_{k-1}}
  \left\|A_{\pi_{k-1}}^{\rm QVI}\cdots A_{\pi_0}^{\rm QVI}\right\|_2
  \right)^{1/k},
  \left(
  \max_{\pi_0,\ldots,\pi_{k-2}}
  \left\|A_{\pi_{k-2}}^{\rm QVI}\cdots A_{\pi_0}^{\rm QVI}\right\|_2
  \right)^{1/k}
  \right\}.
\end{align*}
Here $2^{1/k}\to1$. The first term inside the maximum converges to
$\rho(\mathcal A^{\rm QVI})$ by the JSR definition. The second term has the same
limit because
\begin{align*}
&\left(
  \max_{\pi_0,\ldots,\pi_{k-2}}
  \left\|A_{\pi_{k-2}}^{\rm QVI}\cdots A_{\pi_0}^{\rm QVI}\right\|_2
  \right)^{1/k} \\
&\quad=
  \left[
  \left(
  \max_{\pi_0,\ldots,\pi_{k-2}}
  \left\|A_{\pi_{k-2}}^{\rm QVI}\cdots A_{\pi_0}^{\rm QVI}\right\|_2
  \right)^{1/(k-1)}
  \right]^{(k-1)/k}
  \to\rho(\mathcal A^{\rm QVI}),
\end{align*}
since $(k-1)/k\to1$. Thus the extra factor coming from the two-term sum
vanishes in the JSR limit, and we obtain
$\rho(\mathcal A^{\rm HB})\le\rho(\mathcal A^{\rm QVI})$. Combining the two inequalities proves
\begin{align*}
  \rho(\mathcal A^{\rm HB})
  =
  \rho(\mathcal A^{\rm QVI})
\end{align*}
at $\eta=0$.
\end{proof}

\subsection{Proof of Lemma~\ref{lem:projected-hb-zero-momentum}}
\label{app:proof-lem-projected-hb-zero-momentum}

\begin{proof}
The proof is analogous to that of~\Cref{lem:full-hb-zero-momentum} and is omitted.
\end{proof}

\subsection{Proof of Theorem~\ref{thm:hb-global-jsr-improvement}}
\label{app:proof-thm-hb-global-jsr-improvement}

\begin{proof}
By~\Cref{prop:standard_q_jsr}, the standard benchmark is fixed by the
common eigenvector direction. At $\eta=0$, \Cref{lem:projected-hb-zero-momentum} gives
\begin{align*}
  \rho(\bar{\mathcal A}^{\rm HB})
  &=\rho(\bar{\mathcal A}^{\rm QVI}).
\end{align*}
Using the strict gap in~\Cref{eq:projected-ql-gap},
\begin{align*}
  \rho(\bar{\mathcal A}^{\rm HB})
  <\rho(\mathcal A^{\rm QVI}).
\end{align*}
By the finite-family JSR continuity stated in
\Cref{lem:projected-hb-jsr-continuity}, and by the standard continuity of the
JSR for finite matrix families~\citep[Chapter~1]{jungers2009joint},
\begin{align*}
  \exists\eta_0>0\quad
  \forall\eta\in[0,\eta_0):
  \qquad
  \rho(\bar{\mathcal A}^{\rm HB})<\rho(\mathcal A^{\rm QVI}).
\end{align*}
Thus, whenever $0<\eta<\eta_0$,
\begin{equation}
  \rho(\bar{\mathcal A}^{\rm HB})<\rho(\mathcal A^{\rm QVI}).
  \label{eq:projected-jsr-condition}
\end{equation}
Here the left-hand side is the JSR of the momentum-augmented projected
heavy-ball family in~\Cref{eq:projected-hb-family}; equivalently, it is the JSR
of the projected lag system
\begin{align*}
  \begin{bmatrix}z_{k+1}\\z_k\end{bmatrix}
  =
  \begin{bmatrix}
    \bar A_{\mu_k}^{\rm QVI}+\eta I&-\eta I\\
    I&0
  \end{bmatrix}
  \begin{bmatrix}z_k\\z_{k-1}\end{bmatrix},
\end{align*}
whose first block contains both the projected standard term
$\bar A_{\mu_k}^{\rm QVI}z_k$ and the momentum contribution
$+\eta z_k-\eta z_{k-1}$.
If in addition
\begin{align*}
  0<\eta<\rho(\mathcal A^{\rm QVI})^2,
\end{align*}
then~\Cref{prop:common_eigenvector_acceleration} gives
\begin{align*}
  \rho(C_\eta^-)<\rho(\mathcal A^{\rm QVI}).
\end{align*}
Combining this inequality with~\Cref{eq:projected-jsr-condition} and
\Cref{lem:hb-common-eigenvector-projected-jsr-decomposition} yields
\begin{align*}
  \rho(\mathcal A^{\rm HB}) =\max\left\{\rho(C_\eta^-),\rho(\bar{\mathcal A}^{\rm HB})\right\}<\rho(\mathcal A^{\rm QVI}).
\end{align*}
Therefore, it follows that $\rho(\mathcal A^{\rm HB}) <\rho(\mathcal A^{\rm QVI})$.
For every admissible stochastic selector $\mu_k$, one gets
\begin{align*}
  A_{\mu_k}^{\rm HB}\in\co(\mathcal A^{\rm HB}),
  \qquad
  \rho(\co(\mathcal A^{\rm HB}))=\rho(\mathcal A^{\rm HB})<1,
\end{align*}
and~\Cref{lem:basic_jsr_convergence} gives uniform exponential convergence.
\end{proof}

\section{Heavy-Ball Q-Learning with Linear Function Approximation}
\label{app:hb-lfa}

\subsection{Linear-Approximation Setup}
\label{app:lfa-setup}

This subsection collects the linear function approximation (LFA) notation used in the projected Q-value iteration (PQVI) extension.
Let $\Phi\in\R^{n\times m}$ be a feature matrix. Its
row corresponding to $(s,a)$ is $\phi(s,a)^\top$, where
$\phi(s,a)\in\R^m$. The LFA representation of the Q-function is $Q_\theta:=\Phi\theta$. For an LFA parameter $\theta$, define the corresponding greedy value vector by
\begin{align*}
V_\theta(s)&:=\max_{a\in\mathcal A}\phi(s,a)^\top\theta,\\
V_\theta&:=(V_\theta(1),\ldots,V_\theta(|\mathcal S|))^\top.
\end{align*}
We use $d$ to denote a state-action sampling distribution on
$\mathcal S\times\mathcal A$. In the i.i.d.\ observation model, $d$ is the
sampling distribution of $(s_k,a_k)$; in the Markovian observation model, $d$ is
the stationary state-action distribution of the behavior-induced chain. Define $D:=\diag(d(s,a))_{(s,a)\in\mathcal S\times\mathcal A}$. The feature and sampling conditions used in Appendix~\ref{app:pqvi-lfa} are
collected in the following standing assumption. These conditions ensure that the
projected residual and the Gram matrix used later are well defined in the stated
coordinates.
\begin{assumption}
\label{ass:feature_sampling}
The feature matrix $\Phi$ has full column rank, and the sampling distribution
has full support: $d(s,a)>0$ for every
$(s,a)\in\mathcal S\times\mathcal A$. Equivalently, the diagonal sampling
matrix satisfies $D\succ0$.
\end{assumption}
This assumption also gives $\Phi^\top D\Phi\succ0$.
The SLS model above will be used to analyze Bellman updates by
viewing each realization of the maximum operator as a policy-selected mode. To
make this passage from the Bellman maximum to a switching mode precise, we use
one final setup fact. It says that the difference between two greedy value
vectors is itself a value vector generated by a suitable stochastic policy
applied to the parameter difference. This is the device that converts the
Bellman maximum into a switching mode without changing the parameter space.
\begin{lemma}
\label{lem:stochastic_policy_linearization}
For every pair $\theta,\bar\theta\in\R^m$, there exists a stochastic policy
$\mu_{\theta,\bar\theta}$ such that
\begin{align*}
V_\theta-V_{\bar\theta}
=
\Pi^{\mu_{\theta,\bar\theta}}\Phi(\theta-\bar\theta).
\end{align*}
\end{lemma}
\begin{proof}
Fix a state $s$ and define
\begin{align*}
  u_a:=\phi(s,a)^\top(\theta-\bar\theta),
  \qquad a\in\mathcal A.
\end{align*}
The scalar difference
\begin{align*}
  \max_{a\in\mathcal A}\phi(s,a)^\top\theta
  -\max_{a\in\mathcal A}\phi(s,a)^\top\bar\theta
\end{align*}
lies between $\min_a u_a$ and $\max_a u_a$. Hence it can be written as a convex
combination of the numbers $\{u_a:a\in\mathcal A\}$. Choose one such convex
combination and denote its weights by $\mu_{\theta,\bar\theta}(s)\in\Delta_{|\mathcal A|}$.
Doing this independently for each state gives a stochastic policy
$\mu_{\theta,\bar\theta}$ satisfying
\begin{align*}
  V_\theta(s)-V_{\bar\theta}(s)
  =\sum_{a\in\mathcal A}\mu_{\theta,\bar\theta}(a\mid s)\,
  \phi(s,a)^\top(\theta-\bar\theta),
\end{align*}
for every $s\in\mathcal S$. Stacking the state-wise identities gives
$V_\theta-V_{\bar\theta}=\Pi^{\mu_{\theta,\bar\theta}}\Phi(\theta-\bar\theta)$.
\end{proof}

\subsection{Projected Q-Value Iteration with Linear Function Approximation}
\label{sec:linear_fa_projected_error_system}
\label{app:pqvi-lfa}

This subsection considers projected Q-value iteration (PQVI) with LFA, following the classical projected-equation viewpoint in approximate dynamic programming~\citep{bertsekas1996neuro,tsitsiklis1997analysis}. We add a heavy-ball momentum term to PQVI and study its convergence and acceleration through the same projected-error logic introduced in~\Cref{sec:standard_q_baseline}: the PQVI SLS is identified first, and its JSR is used as the benchmark before the heavy-ball lag is added. To this end, let
\begin{align*}
  Q_\theta:=\Phi\theta,
  \qquad
  \Phi\in\R^{n\times m},
\end{align*}
where $\Phi$ has full column rank. Let $D\succ0$ be a diagonal weighting matrix
and define the $D$-weighted projection
\begin{align*}
  \Pi_D:=\Phi(\Phi^\top D\Phi)^{-1}\Phi^\top D.
\end{align*}
The standard PQVI recursion is written as
\begin{align*}
  \theta_{k+1}
  =\theta_k
  +\alpha(\Phi^\top D\Phi)^{-1}\Phi^\top D\{F(\Phi\theta_k)-\Phi\theta_k\}.
\end{align*}
\begin{lemma}
\label{lem:linear-fa-switching-systems}
Let $\theta^\star$ be a projected Bellman fixed point, i.e.,
\begin{align*}
  \Phi\theta^\star=\Pi_D F(\Phi\theta^\star).
\end{align*}
For any stochastic policy $\mu$, let us define the PQVI mode
\begin{align*}
  A_\mu^{\rm LQVI}:=(1-\alpha)I_m
  +\alpha\gamma(\Phi^\top D\Phi)^{-1}\Phi^\top D P\Pi^\mu\Phi.
\end{align*}
Then, the PQVI error system can be written as
\begin{align*}
  \theta_{k+1}-\theta^\star
  =A_{\mu_k}^{\rm LQVI}(\theta_k-\theta^\star),
\end{align*}
where $\mu_k$ is a stochastic Bellman-difference selector from~\Cref{eq:bellman-difference-selector-general}. The corresponding SLS family is
\begin{align*}
  \mathcal A^{\rm LQVI}:=\left\{A_\pi^{\rm LQVI}:\pi\in\Theta\right\}.
\end{align*}
For every stochastic selector $\mu_k$, one has
$A_{\mu_k}^{\rm LQVI}\in\co(\mathcal A^{\rm LQVI})$.
\end{lemma}
\begin{proof}
Because $\Phi$ has full column rank, the projected fixed-point identity implies
\begin{align*}
  \theta^\star=(\Phi^\top D\Phi)^{-1}\Phi^\top DF(\Phi\theta^\star).
\end{align*}
For each $k$, the selector identity~\Cref{eq:bellman-difference-selector-general}
gives a stochastic policy $\mu_k$ such that
\begin{align*}
  F(\Phi\theta_k)-F(\Phi\theta^\star)
  =\gamma P\Pi^{\mu_k}\Phi(\theta_k-\theta^\star).
\end{align*}
Subtracting the projected fixed-point identity from the standard coefficient
recursion yields
\begin{align*}
  \theta_{k+1}-\theta^\star
  &=(1-\alpha)(\theta_k-\theta^\star)\\
  &\quad+\alpha(\Phi^\top D\Phi)^{-1}\Phi^\top D
    \{F(\Phi\theta_k)-F(\Phi\theta^\star)\}\\
  &=A_{\mu_k}^{\rm LQVI}(\theta_k-\theta^\star).
\end{align*}
The convex-hull statement follows from the linearity of $A_\mu^{\rm LQVI}$ in
$P\Pi^\mu$ and from~\Cref{eq:policy-convex-hull}.
\end{proof}

By adding the momentum term, the corresponding heavy-ball PQVI is
\begin{equation}
\begin{aligned}
\theta_{k+1}
&=\theta_k
  +\alpha(\Phi^\top D\Phi)^{-1}\Phi^\top D
    \{F(\Phi\theta_k)-\Phi\theta_k\}\\
&\quad+\eta(\theta_k-\theta_{k-1}).
\end{aligned}
\label{eq:linear-fa-hb-update}
\end{equation}
When $\Phi=I$, the factor $(\Phi^\top D\Phi)^{-1}\Phi^\top D$ becomes $I$, so
\Cref{eq:linear-fa-hb-update} recovers the tabular heavy-ball recursion in
\Cref{eq:simple-heavy-ball}. We next introduce the corresponding SLS model for the heavy-ball recursion.

\begin{lemma}
\label{lem:linear-fa-heavy-ball-switching-system}
Let $\theta^\star$ be a projected Bellman fixed point,
\begin{align*}
  \Phi\theta^\star=\Pi_D F(\Phi\theta^\star).
\end{align*}
For the heavy-ball recursion~\Cref{eq:linear-fa-hb-update}, define
\begin{equation*}
A_\mu^{\rm LHB}:=
  \begin{bmatrix}
    (1-\alpha+\eta)I_m
    +\alpha\gamma(\Phi^\top D\Phi)^{-1}\Phi^\top D P\Pi^\mu\Phi
    & -\eta I_m\\
    I_m&0
  \end{bmatrix}.
\end{equation*}
Then the heavy-ball PQVI error system can be written as
\begin{align*}
  \begin{bmatrix}\theta_{k+1}-\theta^\star\\\theta_k-\theta^\star\end{bmatrix}
  =A_{\mu_k}^{\rm LHB}
  \begin{bmatrix}\theta_k-\theta^\star\\\theta_{k-1}-\theta^\star\end{bmatrix},
\end{align*}
where $\mu_k$ is a stochastic Bellman-difference selector from~\Cref{eq:bellman-difference-selector-general}. The corresponding heavy-ball SLS family is
\begin{align*}
  \mathcal A^{\rm LHB}:=\left\{A_\pi^{\rm LHB}:\pi\in\Theta\right\}.
\end{align*}
For every stochastic selector $\mu_k$, one has
$A_{\mu_k}^{\rm LHB}\in\co(\mathcal A^{\rm LHB})$.
\end{lemma}
\begin{proof}
Because $\Phi$ has full column rank, the projected fixed-point identity implies
\begin{align*}
  \theta^\star=(\Phi^\top D\Phi)^{-1}\Phi^\top DF(\Phi\theta^\star).
\end{align*}
For each $k$, the selector identity~\Cref{eq:bellman-difference-selector-general}
gives a stochastic policy $\mu_k$ such that
\begin{align*}
  F(\Phi\theta_k)-F(\Phi\theta^\star)
  =\gamma P\Pi^{\mu_k}\Phi(\theta_k-\theta^\star).
\end{align*}
Subtracting the projected fixed-point identity from
\Cref{eq:linear-fa-hb-update} gives
\begin{align*}
\theta_{k+1}-\theta^\star
  &=(1-\alpha+\eta)(\theta_k-\theta^\star)\\
  &\quad+\alpha\gamma(\Phi^\top D\Phi)^{-1}\Phi^\top D\\
  &\qquad\times P\Pi^{\mu_k}\Phi(\theta_k-\theta^\star)\\
  &\quad-\eta(\theta_{k-1}-\theta^\star).
\end{align*}
Stacking this identity with the lag identity gives the stated heavy-ball error
system. The convex-hull statement follows from the linearity of $A_\mu^{\rm LHB}$
in $P\Pi^\mu$ and from~\Cref{eq:policy-convex-hull}.
\end{proof}

Unlike in the tabular case, when linear function approximation is used, it is difficult to apply the common-eigenvector argument or a linear-system recursion argument directly on an invariant subspace. Therefore, we introduce the following additional assumption.

\begin{assumption}
\label{ass:constant-feature}
The feature space contains the constant vector: there exists
$c\in\R^m$ such that
\begin{equation*}
  \Phi c=\mathbf{1}.
\end{equation*}
\end{assumption}
This assumption is needed because, unlike the tabular case, the constant vector need not belong to the span of the chosen features. The assumption is not overly restrictive in many approximate value-function models, because a constant feature is commonly included to represent offsets or baselines; if it is absent, it can be appended as one additional feature without changing the projected-error construction used below. Under this condition, the feature-space coefficient $c$ represents the common constant direction.
Under~\Cref{ass:constant-feature}, for every stochastic policy $\mu$,
\begin{align}
  &(\Phi^\top D\Phi)^{-1}\Phi^\top D P\Pi^\mu\Phi c \notag\\
  &\qquad=(\Phi^\top D\Phi)^{-1}\Phi^\top D P\Pi^\mu\mathbf{1} \notag\\
  &\qquad=(\Phi^\top D\Phi)^{-1}\Phi^\top D\mathbf{1} \notag\\
  &\qquad=(\Phi^\top D\Phi)^{-1}\Phi^\top D\Phi c
  =c.
  \label{eq:feature-constant-invariant}
\end{align}
Hence $\operatorname{span}\left\{c\right\}$ is a common invariant subspace of the induced feature-space
family. The role of the tabular constant direction is now played by the
coefficient direction $c$.

Let $U_c\in\R^{m\times(m-1)}$ have orthonormal columns spanning
$\operatorname{span}\left\{c\right\}^\perp$, so that $U_c^\top U_c=I_{m-1}$ and $U_c^\top c=0$.
Define the projected standard feature-space family
\begin{align*}
\bar A_\pi^{\rm LQVI}&:=U_c^\top A_\pi^{\rm LQVI}U_c,\\
\bar{\mathcal A}^{\rm LQVI}&:=
  \left\{\bar A_\pi^{\rm LQVI}:\pi\in\Theta\right\}.
\end{align*}
and the projected heavy-ball feature-space family
\begin{align*}
\bar A_\pi^{\rm LHB}&:=
  \begin{bmatrix}
    \bar A_\pi^{\rm LQVI}+\eta I_{m-1}&-\eta I_{m-1}\\
    I_{m-1}&0
  \end{bmatrix},\\
\bar{\mathcal A}^{\rm LHB}&:=
  \left\{\bar A_\pi^{\rm LHB}:\pi\in\Theta\right\}.
\end{align*}
With these projected feature-space matrices fixed, the next lemma gives the
standard projected coefficient-error recursion.
\begin{lemma}
\label{lem:linear-fa-standard-projected-error-system}
Assume~\Cref{ass:constant-feature} holds. For the standard projected coefficient
recursion in~\Cref{lem:linear-fa-switching-systems}, the projected coefficient
error satisfies
\begin{align*}
  U_c^\top(\theta_{k+1}-\theta^\star)
  =\bar A_{\mu_k}^{\rm LQVI}U_c^\top(\theta_k-\theta^\star),
\end{align*}
where $\mu_k$ is a stochastic Bellman-difference selector.
\end{lemma}
\begin{proof}
For every stochastic policy $\mu$, \Cref{eq:feature-constant-invariant} gives
\begin{align*}
  A_\mu^{\rm LQVI}c=\bigl(1-\alpha(1-\gamma)\bigr)c.
\end{align*}
Since the columns of $U_c$ form an orthonormal basis of
$\operatorname{span}\{c\}^\perp$,
\begin{align*}
  I-U_cU_c^\top=\frac{cc^\top}{\left\lVert c\right\rVert_2^2}.
\end{align*}
Hence
\begin{align*}
  U_c^\top A_{\mu_k}^{\rm LQVI}(I-U_cU_c^\top)
  &=U_c^\top A_{\mu_k}^{\rm LQVI}
    \frac{cc^\top}{\left\lVert c\right\rVert_2^2}\\
  &=\bigl(1-\alpha(1-\gamma)\bigr)
    U_c^\top c\frac{c^\top}{\left\lVert c\right\rVert_2^2}=0.
\end{align*}
Using the coefficient error system from~\Cref{lem:linear-fa-switching-systems},
we obtain
\begin{align*}
  U_c^\top(\theta_{k+1}-\theta^\star)
  &=U_c^\top A_{\mu_k}^{\rm LQVI}(\theta_k-\theta^\star)\\
  &=U_c^\top A_{\mu_k}^{\rm LQVI}U_cU_c^\top(\theta_k-\theta^\star)\\
  &=\bar A_{\mu_k}^{\rm LQVI}U_c^\top(\theta_k-\theta^\star).
\end{align*}
\end{proof}

The heavy-ball analogue follows by applying the same projection after augmenting
with one lag.
\begin{lemma}
\label{lem:linear-fa-heavy-ball-projected-error-system}
Assume~\Cref{ass:constant-feature} holds. For the heavy-ball coefficient
recursion in~\Cref{lem:linear-fa-heavy-ball-switching-system}, the projected
heavy-ball error system is
\begin{equation*}
  \begin{bmatrix}
    U_c^\top(\theta_{k+1}-\theta^\star)\\
    U_c^\top(\theta_k-\theta^\star)
  \end{bmatrix}
  =\bar A_{\mu_k}^{\rm LHB}
  \begin{bmatrix}
    U_c^\top(\theta_k-\theta^\star)\\
    U_c^\top(\theta_{k-1}-\theta^\star)
  \end{bmatrix},
\end{equation*}
where $\mu_k$ is a stochastic Bellman-difference selector.
\end{lemma}
\begin{proof}
The heavy-ball coefficient error system can be written as
\begin{align*}
  \theta_{k+1}-\theta^\star
  &=\left(A_{\mu_k}^{\rm LQVI}+\eta I_m\right)(\theta_k-\theta^\star)
    -\eta(\theta_{k-1}-\theta^\star).
\end{align*}
Applying $U_c^\top$ and using the projection identity proved in
\Cref{lem:linear-fa-standard-projected-error-system} yields
\begin{equation*}
\begin{aligned}
&U_c^\top(\theta_{k+1}-\theta^\star)\\
&\quad=U_c^\top A_{\mu_k}^{\rm LQVI}(\theta_k-\theta^\star)\\
&\qquad+\eta U_c^\top(\theta_k-\theta^\star)
  -\eta U_c^\top(\theta_{k-1}-\theta^\star)\\
&\quad=\bar A_{\mu_k}^{\rm LQVI}U_c^\top(\theta_k-\theta^\star)\\
&\qquad+\eta U_c^\top(\theta_k-\theta^\star)
  -\eta U_c^\top(\theta_{k-1}-\theta^\star)\\
&\quad=\left(\bar A_{\mu_k}^{\rm LQVI}+\eta I_{m-1}\right)
  U_c^\top(\theta_k-\theta^\star)\\
&\qquad-\eta U_c^\top(\theta_{k-1}-\theta^\star).
\end{aligned}
\end{equation*}
Stacking this identity with the lag identity gives the system with
$\bar A_{\mu_k}^{\rm LHB}$.
\end{proof}

As in the tabular case, stochastic-policy selectors are contained in the convex
hull of the deterministic projected families, and the JSR is unchanged by taking
this convex hull.

\begin{lemma}
\label{lem:linear-fa-standard-jsr-decomposition}
Suppose~\Cref{ass:constant-feature} holds. In the orthonormal basis
$[c/\left\lVert c \right\rVert_2\ \ U_c]$, every $A_\pi^{\rm LQVI}$ is block upper triangular:
\begin{equation}
\begin{aligned}
&\begin{bmatrix}
    (c/\|c\|_2)^\top\\
    U_c^\top
  \end{bmatrix}
  A_\pi^{\rm LQVI}
  \begin{bmatrix}
    c/\|c\|_2&U_c
  \end{bmatrix}\\
&\quad=
  \begin{bmatrix}
    1-\alpha(1-\gamma)&\|c\|_2^{-1}c^\top A_\pi^{\rm LQVI}U_c\\
    0&\bar A_\pi^{\rm LQVI}
  \end{bmatrix}.
\end{aligned}
\label{eq:linear-fa-standard-coordinate-block}
\end{equation}
Consequently,
\begin{align*}
  \rho(\mathcal A^{\rm LQVI})
  =\max\left\{\bigl(1-\alpha(1-\gamma)\bigr),\rho(\bar{\mathcal A}^{\rm LQVI})\right\}.
\end{align*}
\end{lemma}
\begin{proof}
By~\Cref{eq:feature-constant-invariant}, for every deterministic policy $\pi$,
\begin{align*}
  (\Phi^\top D\Phi)^{-1}\Phi^\top D P\Pi^\pi\Phi c=c.
\end{align*}
Therefore
\begin{align*}
  A_\pi^{\rm LQVI}c
  &=(1-\alpha)c
    +\alpha\gamma(\Phi^\top D\Phi)^{-1}\Phi^\top D P\Pi^\pi\Phi c\\
  &=\bigl(1-\alpha(1-\gamma)\bigr)c.
\end{align*}
Dividing by $\left\lVert c \right\rVert_2$ gives
\begin{align*}
  A_\pi^{\rm LQVI}\frac{c}{\left\lVert c \right\rVert_2}
  =\bigl(1-\alpha(1-\gamma)\bigr)\frac{c}{\left\lVert c \right\rVert_2}.
\end{align*}
Since $U_c^\top c=0$, the lower-left block in the coordinate matrix in~\Cref{eq:linear-fa-standard-coordinate-block} is
\begin{align*}
  U_c^\top A_\pi^{\rm LQVI}\frac{c}{\left\lVert c \right\rVert_2}
  =\bigl(1-\alpha(1-\gamma)\bigr)U_c^\top\frac{c}{\left\lVert c \right\rVert_2}=0.
\end{align*}
The upper-left block is
\begin{align*}
  \left(\frac{c}{\left\lVert c \right\rVert_2}\right)^\top
  A_\pi^{\rm LQVI}\frac{c}{\left\lVert c \right\rVert_2}
  =1-\alpha(1-\gamma),
\end{align*}
and the lower-right block is exactly
\begin{align*}
  U_c^\top A_\pi^{\rm LQVI}U_c=\bar A_\pi^{\rm LQVI}.
\end{align*}
This proves the block upper triangular form. Products of the transformed matrices
remain block upper triangular. For any word $(\pi_0,\ldots,\pi_{k-1})$, the two
diagonal blocks are
\begin{align*}
  \bigl(1-\alpha(1-\gamma)\bigr)^k
  \quad\text{and}\quad
  \bar A_{\pi_{k-1}}^{\rm LQVI}\cdots\bar A_{\pi_0}^{\rm LQVI}.
\end{align*}
Thus the block triangular JSR decomposition in~\Cref{lem:block-triangular-jsr}
gives the maximum of the scalar-block JSR and the projected-block JSR. The
standing discounted and relaxation assumptions make the scalar-block JSR equal to
$1-\alpha(1-\gamma)$, and the claimed equality follows.
\end{proof}

This decomposition gives an exact criterion for when the full standard
feature-space JSR is determined by the constant block.

\begin{lemma}
\label{lem:linear-fa-standard-jsr-equality}
Suppose~\Cref{ass:constant-feature} holds. Then
\begin{align*}
  \rho(\mathcal A^{\rm LQVI})=1-\alpha(1-\gamma)
\end{align*}
if and only if
\begin{align*}
  \rho(\bar{\mathcal A}^{\rm LQVI})\le 1-\alpha(1-\gamma).
\end{align*}
\end{lemma}
\begin{proof}
By~\Cref{lem:linear-fa-standard-jsr-decomposition},
\begin{align*}
  \rho(\mathcal A^{\rm LQVI})
  =\max\left\{1-\alpha(1-\gamma),\rho(\bar{\mathcal A}^{\rm LQVI})\right\}.
\end{align*}
The equality above holds exactly when the projected block is no larger than
the scalar constant block, which gives the stated equivalence.
\end{proof}

The next lemma states the strict-gap version needed in the final approximation
certificate.

\begin{lemma}
\label{lem:linear-fa-strict-gap-implies-standard-jsr}
Suppose~\Cref{ass:constant-feature} holds. If
\begin{align*}
  \rho(\bar{\mathcal A}^{\rm LQVI})<\rho(\mathcal A^{\rm LQVI}),
\end{align*}
then
\begin{align*}
  \rho(\mathcal A^{\rm LQVI})=1-\alpha(1-\gamma).
\end{align*}
\end{lemma}
\begin{proof}
By~\Cref{lem:linear-fa-standard-jsr-decomposition},
\begin{align*}
  \rho(\mathcal A^{\rm LQVI})
  =\max\left\{1-\alpha(1-\gamma),\rho(\bar{\mathcal A}^{\rm LQVI})\right\}.
\end{align*}
If $\rho(\bar{\mathcal A}^{\rm LQVI})<\rho(\mathcal A^{\rm LQVI})$, the maximum
cannot be attained by the projected term. Therefore it must be attained by
$1-\alpha(1-\gamma)$.
\end{proof}

The final theorem transfers the tabular constant/projection argument to the
projected LFA setting.

\begin{theorem}
\label{thm:linear-fa-jsr-improvement}
Suppose~\Cref{ass:constant-feature} holds. If
\begin{equation}
  \rho(\bar{\mathcal A}^{\rm LQVI})<\rho(\mathcal A^{\rm LQVI}),
  \label{eq:linear-fa-projected-gap}
\end{equation}
then there exists $\eta_0>0$ such that every
\begin{equation}
  0<\eta<\min\left\{\bigl(1-\alpha(1-\gamma)\bigr)^2,\eta_0\right\}
  \label{eq:linear-fa-eta-condition}
\end{equation}
satisfies
\begin{equation}
  \rho(\bar{\mathcal A}^{\rm LHB})<1-\alpha(1-\gamma)
  \label{eq:linear-fa-projected-jsr-less}
\end{equation}
and
\begin{equation}
  \rho(\mathcal A^{\rm LHB})
  <\rho(\mathcal A^{\rm LQVI})=\bigl(1-\alpha(1-\gamma)\bigr).
  \label{eq:linear-fa-jsr-less}
\end{equation}
\end{theorem}
\begin{proof}
By~\Cref{lem:linear-fa-strict-gap-implies-standard-jsr},
\begin{align*}
  \rho(\mathcal A^{\rm LQVI})=\bigl(1-\alpha(1-\gamma)\bigr).
\end{align*}
At $\eta=0$, the projected heavy-ball matrices are
\begin{align*}
  \bar A_\pi^{\rm LHB}=
  \begin{bmatrix}
    \bar A_\pi^{\rm LQVI}&0\\
    I_{m-1}&0
  \end{bmatrix},
\end{align*}
and the JSR equals $\rho(\bar{\mathcal A}^{\rm LQVI})$. Since the JSR of a finite
family depends continuously on the matrices, the strict gap in
\Cref{eq:linear-fa-projected-gap} persists for all sufficiently small positive
$\eta$. Thus there exists $\eta_0>0$ such that
\Cref{eq:linear-fa-projected-jsr-less} holds whenever $0<\eta<\eta_0$.

In the augmented basis induced by $[c/\left\lVert c \right\rVert_2\ \ U_c]$, each
$A_\pi^{\rm LHB}$ is block upper triangular. One diagonal block is the same
common-eigenvector companion matrix $C_\eta^-$ from~\Cref{eq:common-eigenvector-companion}; the other
diagonal block belongs to $\bar{\mathcal A}^{\rm LHB}$. The interval
\Cref{eq:linear-fa-eta-condition} gives
$\rho(C_\eta^-)<1-\alpha(1-\gamma)$ by~\Cref{prop:common_eigenvector_acceleration},
while~\Cref{eq:linear-fa-projected-jsr-less} gives the same strict bound on the
projected feature-space block. By the block triangular JSR decomposition in~\Cref{lem:block-triangular-jsr},
the JSR of this block upper triangular SLS family is the maximum of the
JSRs of its diagonal block families. Therefore \Cref{eq:linear-fa-jsr-less}
follows.
\end{proof}

\end{document}